\documentclass{article}

\usepackage{amsthm}
\usepackage{amssymb}
\usepackage{amsmath}
\usepackage{amscd}
\usepackage{makeidx}
\usepackage{fancyhdr}
\usepackage[all]{xy}
\usepackage{supertabular}

\newcommand{\De}{\Delta}
\newcommand{\Ga}{\Gamma}
\newcommand{\la}{\lambda}

\newcommand{\FF}{\mathbb{F}}
\newcommand{\QQ}{\mathbb{Q}}
\newcommand{\ZZ}{\mathbb{Z}}

\newcommand{\Gal}{\mathop{\rm Gal}}

\newcommand{\ok}{$\emptyset$}

\newcommand{\te}[2]{$\{{#1}\}_{#2}$}

\newcommand{\sopra}[2]{\genfrac{}{}{0pt}{1}{#1}{#2}}
\def\smallmat#1#2#3#4{
    \left(\sopra{#1}{#3}\sopra{#2}{#4}\right)}

\newtheorem{definition}{Definition}[section]
\newtheorem{theorem}[definition]{Theorem}
\newtheorem{remark}[definition]{Remark}
\newtheorem{proposition}[definition]{Proposition}

\begin{document}

\markboth{T. Boggio, A. Mori}{Power multiples in binary recurrence sequences}

\title{Power multiples in binary recurrence sequences: \\ an approach by congruences}
\author{\Large Teresa Boggio, Andrea Mori\\
             \large Dipartimento di Matematica\\
             Universit\`a di Torino\\ Torino, Italia}
             
 \date{}
 
 \maketitle

\begin{abstract}
We introduce an elementary congruence-based procedure to look for $q$-th power multiples in 
arbitrary binary recurrence sequences ($q\geq3$). The procedure allows to prove that no
 such multiples exist in many instances.

\end{abstract}


\section{Introduction and result}
Let $u,v,A,B\in\ZZ$. The ($\ZZ$-valued) binary recurrence sequence with initial values $u$, $v$ and coefficients $A$, $B$ is the sequence $\{G_n\}_{n\geq0}$ defined recursively as
\begin{equation}
G_0=u,\quad G_1=v,\quad G_{n+2}=AG_{n+1}+BG_n \textrm{ for all $n\geq0$}.
\label{seqdef}
\end{equation}
The discriminant of the sequence \eqref{seqdef} is the integer $\De=A^2+4B\neq0$.
An equivalent description is
\begin{equation}
\left(\begin{array}{c}G_{n+2} \\G_{n+1}\end{array}\right)=
\left(\begin{array}{cc}A & B \\1 & 0\end{array}\right)
\left(\begin{array}{c}G_{n+1} \\G_{n}\end{array}\right),
\label{matdef}
\end{equation}
i.÷e. $\left(\sopra{G_{n+1}}{G_n}\right)={\smallmat AB10}^n\left(\sopra{G_{1}}{G_0}\right)$,
for all $n\geq0$. Let $K$ be the smallest extension of $\QQ$ containing the eigenvalues  
$\{\la_1,\la_2\}$ of the matrix $\smallmat AB10$ and denote ${\cal O}_K$ its ring of integers.
 Either $K=\QQ$ or $K$ is quadratic, $K=\QQ(\sqrt{\De})$, and in the latter case write 
$\Gal(K/\QQ)=\langle\tau\rangle$.
The sequence \eqref{seqdef} is called non-degenerate if $\la_1/\la_2$ is not a root of $1$.
Also, if $\la_1\neq\la_2$ the sequence is a generalized power sum with constant coefficients, namely
$$
G_n=g_1\la_1^n+g_2\la_2^n, \qquad
\textrm{where $g_1=\frac{G_1-\la_2G_0}{\la_1-\la_2}$, $g_2=\frac{\la_1G_0-G_1}{\la_1-\la_2}$.}
$$

A sequence with values in $\ZZ$ can be \lq\lq followed\rq\rq\  looking for integers with special 
interesting arithmetic properties (Ribenboim \cite{Ri} likens this to picking wild flowers during 
a walk in the countryside). In this note we deal with the equation
\begin{equation}
G_n=kx^q
\label{mainprob}
\end{equation}
where $0\neq k\in\ZZ$ is a fixed constant and $q\geq3$. As usual, we may and shall 
assume that $q$ is a prime number. 

By relating it to Baker's theory of linear forms in logarithms,  Peth\"o \cite{Pe} and 
Shorey and Stewart \cite{SS} proved independently that \eqref{mainprob} has, 
under some mild conditions on the sequence, only finitely many solutions 
$(n,G_n,x,q)$. Peth\"o's precise version of the result is the following.

\begin{theorem}\label{th:petho}
     Let $\{G_n\}$ be a binary recurrence sequence with coprime non-zero coefficients 
     $A$ and $B$ such that $(G_0,G_1)\neq(0,0)$, $A^2\neq -jB$ for $j\in\{1,2,3,4\}$ and  
     $G_1^2-AG_0G_1-BG_0^2\neq0$. Let $\cal P$ be a finite set of primes and let $\cal S$ 
     be the set of integers divisible only by primes in $\cal P$.
     Then, there exists an effective constant $C=C(A,B,G_0,G_1,\cal P)$ such that if 
     $G_n=kx^q$ with $k\in\cal S$ and $|x|>1$ then $\max(n,|G_n|,|x|,q)<C$.
\end{theorem}

\begin{remark}
     When the sequence $\{G_n\}$ is non-degenerate and $k$ is any fixed integer, the finiteness 
     of the number of solutions of $G_n=k$ (i.÷e. the $x$-trivial solutions of \eqref{mainprob}) 
     follows from the Skolem-Mahler-Lech theorem, \cite[\S2.1]{EPSW}, which is independent 
     of Baker's theory.
\end{remark}

Although theorem  \ref{th:petho} reduces in principle the problem of finding all the solutions 
of \eqref{mainprob} to a finite amount of computations, from a practical point of view the 
possibility of using brute force is illusory since the constant $C$ is huge. 
Following the steps of the proof of theorem \ref{th:petho} in the arguably simplest case 
of the Fibonacci sequence $\{F_n\}$ (obtained for $u=0$, $v=1$, $A=B=1$) 
the first author \cite{Bo} found that for a solution of \eqref{mainprob} with $k=1$ the bounds are
$q\leq192^{1203}$, and $|x|\leq e^{5^{80(4q!+1)(4q!+5)}/4q!}$.
Even for a single sequence $\{G_n\}$, the problem of finding a complete solution 
of  \eqref{mainprob} may be far from trivial. For instance, it had been known for a 
while that the only squares and cubes in the Fibonacci sequence are 
$\{F_0=0, F_1=1, F_2=1, F_{12}=144\}$ and 
$\{F_0=0, F_1=1, F_2=1, F_{6}=8\}$ respectively, but to prove that those are the 
only powers, Bugeaud, Mignotte and Siksek \cite{BMS} had to combine the classical 
approach with modular methods similar to those used by Wiles to prove Fermat's last theorem.

Let us fix the exponent $q$. We present an elementary procedure, introduced in \cite{Bo}, to approximate the solutions of \eqref{mainprob} in the following sense.
The procedure outputs a large integer $N=N_q$ and a relatively small set 
${\cal J}\subset\ZZ/N\ZZ$ such that if $G_n$ solves \eqref{mainprob} then 
$\overline{n}=n\bmod N\in\cal J$. The actual computations show that the procedure 
\lq\lq converges\rq\rq\  rather quickly and in many cases yields 
${\cal J}=\emptyset$ showing the absence of solutions for the corresponding equation.

The procedure is explained in section 2 followed by some heuristics in section 3. 
A final section gives a few example of actual computations. 
We test  all non-trivial sequences $\{G_n\}$ with positive parameters $A$ and $B$, 
and non-negative initial values $G_0$ and $G_1$ with $A+B\leq4$ and 
$\max\{G_0,G_1\}\leq9$ up to shift-equivalence (see Definition \ref{df:equiv}). 
There are two kinds of tables. Tables 1 to 6 show the result of running the procedure 
in search of $q$-powers, for $q\in\{3,5,7,11,13, 17\}$. Tables 7 to 12 list the values of 
$k$ for which \eqref{mainprob} with $q=3$ or $q=5$ has no solutions for 
$2\leq k\leq30$ and $q$-power free. In particular, the following result remains proved.

\begin{theorem}
     Let $\{G_n\}$ be a binary recurrence sequence. The equation $G_n=kx^q$ has 
     no solutions in all cases labelled {\ok} in Tables 1 to 6 and for all values $(q,k)$
     listed in Tables 7 to 12 below.
\end{theorem}

An analysis of the tables 1--6 shows that in many cases, up to replacing $N$ 
by a large divisor, the set $\cal J$ consists of just one element, so that up to 
shift-equivalence we may assume that ${\cal J}=\{\overline{0}\}$. 
The following question arises naturally. Suppose that there is a (large) integer 
$N$ such that  a solution of $G_n=kx^q$ can occur only for $n\equiv0\bmod N$. 
Can we obtain further information on the set of solutions from arithmetic properties 
of the triple $(k,q,N)$? In particular, can we deduce the finiteness of the 
number of solutions independently of Baker's theory?


\section{The procedure}
We shall assume that $AB\neq0$. The binary recurrence sequence \eqref{seqdef} extends 
uniquely to a function 
$\ZZ\rightarrow\ZZ[1/B]$ in such a way that the recurrence relation 
$G_{n+2}=AG_{n+1}+BG_n$ remains valid for all $n\in\ZZ$. Namely, set inductively
$$
G_{-n}=-\frac ABG_{-n+1}+\frac1BG_{-n+2}\qquad\textrm{for all $n>0$.}
$$
\begin{definition}\label{df:equiv}
Two extended binary recurrence sequences $\{G_n\}$ and 
$\{G_n^\prime\}$ are called shift-equivalent if there exists 
$k\in\ZZ$ such that $G_n^\prime=G_{n+k}$ for all $n\in\ZZ$.
\end{definition}

\begin{proposition}
  \begin{enumerate}
  \item Two sequences not of the form $\{g\mu^n\}$ are shift-equiva\-lent if and only if they 
            share four equal consecutive terms.
  \item The sequences $\{g\mu^n\}$ and $\{G_n\}$ are shift-equivalent if and only if 
            $G_n=g^\prime\mu^n$ with $g^\prime=g\mu^k$ for some $k\in\ZZ$.
\end{enumerate}
\end{proposition}

\begin{proof}
The sequences $\{G_n\}_{n\in\ZZ}$ and $\{G_n^\prime\}_{n\in\ZZ}$ 
with same parameters $A$ and $B$ are shift-equivalent if and only if they have 
a common segment of length 2, $G^\prime_{r}=G_s$ and $G^\prime_{r+1}=G_{s+1}$ 
for some $r$, $s\in\ZZ$. When $G_{k}^2\neq AG_{k}G_{k-1}+BG_{k-1}^2$ for some (or, equivalently, all) 
$k\in\ZZ$ the parameters $A$ and $B$ can be recovered from the consecutive terms 
$G_{k-1},\cdots,G_{k+2}$ by solving the linear equations
$$
\left\{
\begin{array}{rcl}
G_{k+2} & = & AG_{k+1}+BG_k \\
G_{k+1} & = & AG_{k}+BG_{k-1} 
\end{array}
\right.
$$
This proves part 1 once we observe that the sequences of the form $\{g\mu^n\}$ are precisely those for which 
$G_{k}^2=AG_{k}G_{k-1}+BG_{k-1}^2$.
Part 2 is immediate.
\end{proof}

The previous fact remains true for $R$-valued sequences, where $R$ is any domain of characteristic 
prime to $B$.

\begin{definition}
    Let $\ell$ be a prime number, $(\ell,B)=1$. The reduction modulo $\ell$ of the $\ZZ$-valued 
    binary recurrence sequence \eqref{seqdef} is the sequence $\{\overline{G}_n\}$ where 
    $\overline{G}_n\in\FF_\ell=\ZZ/\ell\ZZ$ is the class of $G_n$.
\end{definition}

The reduced sequence $\{\overline{G}_n\}$ is an $\FF_\ell$-valued binary recurrence 
sequence with parameters $\overline A$ and $\overline B\neq0$ and initial values 
$\overline u$, $\overline v$. Its extension  $\{\overline{G}_n\}_{n\in\ZZ}$ is the 
reduction modulo $\ell$ of the extension $\{G_n\}$. The following very simple
fact is the basis of the procedure.

\begin{proposition}
   Let $\{\overline{G}_n\}$ be an extended  $\FF_\ell$-valued binary recurrence 
   sequence. Then $\{\overline{G}_n\}$ is periodic.
\end{proposition} 

\begin{proof}
    Since there are only a finite number of pairs $(a,b)\in\FF_\ell\times\FF_\ell$, there must be integers 
    $r\neq s$ such that $\overline{G}_r=\overline{G}_s$ and 
    $\overline{G}_{r+1}=\overline{G}_{s+1}$. If $0\neq k=s-r$, an obvious induction shows that
     the sequences $\{\overline{G}_{n}\}$ and $\{\overline{G}_{n+k}\}$ coincide.
\end{proof}

\begin{definition}
For a prime number $\ell$, let $\pi_\ell$ be the minimal period of the extended 
$\FF_\ell$-valued reduced sequence $\{\overline{G}_{n}\}$, i.÷e.
$$
\pi_\ell=\min\left\{k\in\ZZ^{>0}
\textrm{ such that $\overline{G}_{n+k}=\overline{G}_n$ for all $n\in\ZZ$}\right\}.
$$
\end{definition}

\begin{proposition}\label{divide}
  Let $\ell$ be a prime number. The period $\pi_\ell$ is a divisor of
  \begin{enumerate}
  \item $\ell(\ell-1)$, if $\De=0$ or if  $\De$ is not a square in $\ZZ$ with $\ell\mid\De$;
  \item $\ell-1$, if $\De$ is a non-zero square or if $\left(\frac{\De}{\ell}\right)=1$;
  \item $\ell^2-1$, if $\De$ is not a square and $\left(\frac{\De}{\ell}\right)=-1$
\end{enumerate}
\end{proposition}

\begin{proof}
From the description \eqref{matdef}, the period $\pi_\ell$ is the order of the cyclic quotient group
${\langle\overline{M}\rangle}/{\langle\overline{M}\rangle\cap S_{\overline{u},\overline{v}}}$
where $\overline{M}\in\textrm{GL}_2(\FF_\ell)$ is the reduction modulo $\ell$ of 
$M=\smallmat AB10$ and $S_{\overline{u},\overline{v}}$ is the stabilizer of the vector
$\left(\sopra{\overline{u}}{\overline{v}}\right)$ under the tautological action of 
$\textrm{GL}_2(\FF_\ell)$ on $(\FF_\ell)^2$. Thus $\pi_\ell\mid{\rm ord}(\overline{M})$.

If $\De=0$ then $K=\QQ$, $\la_1=\la_2=\la\in\ZZ$ and $M\sim\smallmat{\la}10{\la}$, 
whose order modulo $\ell$ is $\ell(\ell-1)$.

If $\De\neq0$ the eigenvalues are different, so $M\sim\smallmat{\la_1}00{\la_2}$ with 
$\la_1$, $\la_2\in\QQ$ if $\De$ is a square or $\la_2=\tau(\la_1)$ otherwise. 
Hence ${\rm ord}(\overline{M})$ is the least common divisors of the orders of 
$\overline{\la}_1$ and  $\overline{\la}_2$ as elements of 
$({\cal O}_K/\ell{\cal O}_K)^\times$. Thus the other cases follow recalling that 
$$
({\cal O}_K/\ell{\cal O}_K)^\times\simeq
\begin{cases}
   \FF_\ell^\times & \text{if $K=\QQ$}, \\
   \FF_\ell^\times\times  \FF_\ell^\times  & \text{if $K$ quadratic and $\ell$ split}, \\
   \FF_{\ell^2}^\times   & \text{if $K$ quadratic and $\ell$ inert}, \\
   (\FF_\ell[X]/(X^2))^\times   & \text{if $K$ quadratic and $\ell$ ramified}.
\end{cases}
$$
\end{proof}

The procedure goes as follows.
\begin{description}
  \item[Step 1:]  Input the defining data $(u,v,A,B)$, the equation data $(k,q)$ and 
                           fix a cutoff value $C_{\rm off}>0$.
  \item[Step 2:]  Consider the primes $\ell_1<\cdots<\ell_r\leq C_{\rm off}$ satisfying the following 
                            three conditions:
           \begin{enumerate}
                 \item $\ell_i$ does not divide $Bk$ for all $i=1, \dots, r$;
                 \item $\ell_i\equiv1\bmod q$ for all $i=1, \dots, r$;
                 \item if we set $n_1=\pi_{\ell_1}$ and define $n_{i+1}$ for $i=1, \dots r-1$ inductively as 
                          $n_{i+1}=\textrm{lcm}(n_{i},\pi_{\ell_{i+1}})$, then $n_{i+1}/n_i<q$ 
                          for all $i=1, 2, \dots r-1$.
           \end{enumerate}
  \item[Step 3:]  Construct inductively sets ${\cal J}_i\subset\ZZ/n_i\ZZ$ as follows:
          \begin{enumerate}
                 \item ${\cal J}_1=\{\overline{n}\in\ZZ/n_1\ZZ\textrm{ such that }
                           \overline{G}_n/\overline{k}\in(\FF_{\ell_1})^q\}$;
                 \item for $i=1, 2, \dots r-1$, given ${\cal J}_i$ first set 
                          $$
                          {\cal J}_{i+1}^\sharp=\{\overline{n}\in\ZZ/n_{i+1}\ZZ\textrm{ such that }
                          n\bmod{n_i}\in{\cal J}_i \}
                          $$
                          and then let
                          $$
                          {\cal J}_{i+1}={\cal J}_{i+1}^\sharp-
                          \{\overline{n}\textrm{ such that } \overline{G}_n/\overline{k}\notin(\FF_{\ell_{i+1}})^q
                          \}
                          $$
          \end{enumerate}
    \item[Step 4:] If ${\cal J}_{r^\prime}=\emptyset$ for some $r^\prime\leq r$ the procedure stops,
             else let $N=n_r$ and output ${\cal J}={\cal J}_r\subset\ZZ/N\ZZ$. 
\end{description}
The reason for the conditions on the primes $\ell_i$ is the following. The subgroup 
$(\FF_\ell^{\times})^q$ of $q$-powers in the multiplicative group $\FF_\ell^\times$ is proper if and only if 
$q\mid\ell-1$, and in this case consists of $(\ell-1)/q$ elements. Thus, the number of $q$-powers in 
$\FF_\ell$ is $(q+\ell-1)/q$ and on average we can expect that at each step
$$
\left|{\cal J}_{i+1}\right|\cong\frac{q+\ell_{i+1}-1}{q\ell_{i+1}}\left|{\cal J}_{i+1}^\sharp\right|.
$$
Since $|{\cal J}_{i+1}^\sharp|=(n_{i+1}/n_i)|{\cal J}_{i}|$, by forcing $n_{i+1}/n_i\leq q-1$ 
and observing that $\lim_{i\to\infty}\frac{q+\ell_i-1}{q\ell_i}(q-1)<1$ we can expect that eventually  
$|{\cal J}_{i+1}|<|{\cal J}_{i}|$ on average, so that the procedure should eventually produce an empty 
set of indices when the equation \eqref{mainprob} has no solutions.

\begin{remark}
      The necessity of imposing condition 3 in Step 2 makes the procedure unsuited for the case $q=2$.
\end{remark}

\section{Heuristic density estimates}
The support of $n\in\ZZ$ is the set
$\mathop{\rm Supp}(n)=\{\text{$p$ prime such that $p\mid n$}\}$.
Fix an integer $m\geq2$ and let
${\cal P}_m=
\{\text{$\ell$ prime such that $\max(\mathop{\rm Supp}(\pi_\ell))\leq m$}\}$
and
$$
{\cal P}_m^\prime=
\{\text{$\ell$ prime such that $\max(\mathop{\rm Supp}({\rm ord}_\ell(\overline{M})))\leq m$}\}.
$$
Also, let ${\cal P}_{m,q}=\{\text{$\ell\in{\cal P}_m$ such that $\ell\equiv1\bmod q$}\}$
and 
$$
{\cal P}_{m,q}^\prime=\{\text{$\ell\in{\cal P}_m^\prime$ such that $\ell\equiv1\bmod q$}\}.
$$
The sets ${\cal P}_m^\prime$ and ${\cal P}_{m,q}^\prime$ depend on the coefficients $A$ and $B$, 
while the sets ${\cal P}_m$ and ${\cal P}_{m,q}$ depend also on the vector 
$\vec v=\left({{\sopra uv}}\right)\in\ZZ^2$ 
of initial values.
Since $\pi_\ell\mid{\rm ord}_\ell(\overline{M})$, we have that 
${\cal P}_m^\prime\subseteq{\cal P}_m$ and ${\cal P}_{m,q}^\prime\subseteq{\cal P}_{m,q}$.
The primes $\ell_1,\ell_2,\ldots$ of Step 2 are in ${\cal P}_{q-1,q}$.
We shall show that in the case of a non-degenerate binary recurrence sequence with 
non-zero initial vector $\vec v$, a variation 
of the classical Artin heuristics, under the usual independence hypotheses, yields that
the expected density of the sets ${\cal P}_m$, and hence ${\cal P}_{m,q}$, is $0$.

Let assume first that $K=\QQ$ and, for the sake of uniformity of the argument, also that
$\min\{|\la_1|,|\la_2|\}\geq2$. Let $\Sigma_0$ be the finite set of primes containing $2$ 
and the primes dividing $\la_1\la_2$. Consider a prime $\ell\notin\Sigma_0$ 
and write $\ell-1=ab$ where 
$\max\{\mathop{\rm Supp}(a)\}\leq m$ and $\min\{\mathop{\rm Supp}(b)\}> m$. Then 
$(\overline{\la}_1,\overline{\la}_2)\in\FF_\ell^\times\times\FF_\ell^\times$ and
\begin{eqnarray*}
	\max(\mathop{\rm Supp}({\rm ord}_\ell(\overline{M})))\leq m & \Longleftrightarrow &
	\mbox{$\overline{\la}_1$ and $\overline{\la}_2$ are $b$-powers in $\FF_\ell^\times$}  \\
	 & \Longleftrightarrow & \sopra{\displaystyle\hbox{$\overline{\la}_1$ and $\overline{\la}_2$ 
	 are $p^r$-powers in $\FF_\ell^\times$ for}}{\displaystyle\hbox{all primes $p>m$ such that 
	 $p^r\mid\mid\ell-1$}}
\end{eqnarray*}
Since the primes $\ell\equiv1\bmod p^r$ are precisely those that split completely in the 
cyclotomic extension $\QQ\subset\QQ(\mu_{p^r})$, we can rephrase the last condition 
in terms of the extensions in the diagram
\begin{equation}
\label{Qext}
\xymatrix{
\QQ(\mu_{p^{r+1}}) \ar@{-}@/_/[ddr]_{\rm I} \ar@{-}[dr] &  & 
\QQ(\mu_{p^r},\sqrt[p^r]{\la_1},\sqrt[p^r]{\la_2})\ar@{-}@/^/[ddl]^{\rm II} \ar@{-}[dl] \\
 &  \QQ(\mu_{p^r}) & \\
 & \QQ \ar@{-}[u] & 
}.
\end{equation}
Namely, $\overline{\la}_1$ and $\overline{\la}_2$ are  in $(\FF_\ell^\times)^{p^r}$ and 
$p^r\mid\mid\ell-1$ if and only if $\ell\in\Sigma_{p,r}^\prime$, where
$\Sigma_{p,r}^\prime
=\{\text{primes $\ell$ that split completely in II and do not split completely in I}\}$.
By construction, $\Sigma_{p,r}^\prime\cap\Sigma_{p,r^\prime}^\prime$ if 
$r\neq r^\prime$, and if we let  $\Sigma_p^\prime=\cup_{r\geq1}\Sigma_{p,r}^\prime$ then
\begin{equation}
\label{inters}
{\cal P}_m^\prime=\bigcap_{p>m}\Sigma_p^\prime.
\end{equation}
The following proposition is a straightforward application of Kummer's theory
to the situation of diagram \eqref{Qext}.
\begin{proposition}\label{propQ}
Suppose $p\notin\Sigma_0$. Then:
\begin{enumerate}
  \item $[\QQ(\mu_{p^r},\sqrt[p^r]{\la_i}):\QQ(\mu_{p^r})]=p^r$ for $i=1$, $2$;
  \item $\QQ(\mu_{p^r},\sqrt[p^r]{\la_1})\cap\QQ(\mu_{p^r},\sqrt[p^r]{\la_2})=\QQ(\mu_{p^r})$;
  \item $\Gal(\QQ(\mu_{p^r},\sqrt[p^r]{\la_1},\sqrt[p^r]{\la_2})/\QQ(\mu_{p^r}))\simeq
            (\ZZ/p^r\ZZ)^2$;
  \item $\QQ(\mu_{p^{r+1}})\cap\QQ(\mu_{p^r},\sqrt[p^r]{\la_1},\sqrt[p^r]{\la_2})=\QQ(\mu_{p^r})$.
\end{enumerate}
\end{proposition}

In particular, for $p\notin\Sigma_0$ point 4 says that $\Sigma_{p,r}^\prime\neq\emptyset$ 
and by \v{C}ebotarev's theorem the expected density of $\Sigma_{p,r}^\prime$ is
\begin{eqnarray*}
\delta(\Sigma_{p,r}^\prime) & = & \left(1-\frac1{[\QQ(\mu_{p^{r+1}}):\QQ(\mu_{p^r})]}\right)\frac1{[\QQ(\mu_{p^r},\sqrt[p^r]{\la_1},\sqrt[p^r]{\la_2}):\QQ]} \\
  & = & \frac{p-1}p\frac1{p^{3r-1}(p-1)} = \frac1{p^{3r}}
\end{eqnarray*}
so that $\delta(\Sigma_p^\prime)=\sum_{r\geq1}p^{-3r}=1/(p^3-1)$. Applying the independence
assumption to \eqref{inters} yields the expected value
$$
\delta({\cal P}_m^\prime)=
\prod_{\begin{subarray}{c}p>m\\ p\in\Sigma_0\end{subarray}}\delta(\Sigma_p^\prime)
\prod_{\begin{subarray}{c}p>m\\ p\notin\Sigma_0\end{subarray}}\frac1{p^3-1}=0.
$$

Let $\ell\in{\cal P}_m-{\cal P}_m^\prime$, $\ell\notin\Sigma_0$. 
Then $M^{\pi_\ell}\not\equiv I\bmod\ell$ and yet
\begin{equation}
\label{eq:eigen}
M^{\pi_\ell}{\vec v}\equiv{\vec v}\bmod\ell
\end{equation}
In order for this to be possible, the matrix $M^{\pi_\ell}\bmod\ell$ must admit $1$ as an eigenvalue. 
Thus a prime $\ell\notin\Sigma_0$ is in ${\cal P}_m-{\cal P}_m^\prime$ if and only if
 the following two conditions are satisfied.
 \begin{enumerate}
  \item[C1.] Exactly one of the eigenvalues $\la_1$, $\la_2$ is a $b$-power in $\FF_\ell^\times$. 
            Equivalently, exactly one of the eigenvalues $\la_1$, $\la_2$ is a $p^r$-power in 
            $\FF_\ell^\times$ for all $p^r\mid\mid\ell-1$ with $p>m$.
  \item[C2.] If $\la$ is the eigenvalue of condition C1, then ${\vec v}\bmod \ell\in E_\la$ where 
           $E_\la\subset(\ZZ/\ell\ZZ)^2$ is the $\la$-eigenspace of $M\bmod\ell$.
\end{enumerate}
Denote ${\cal P}_m^\flat$ the set of primes satisfying condition C1 only. As above
${\cal P}_m^\flat=\bigcap_{p>m}\Sigma_p$ where $\Sigma_p=\bigcup_{r\geq1}\Sigma_{p,r}$
is a disjoint union with
$$
\Sigma_{p,r}=\left\{
\sopra
{\displaystyle\hbox{$\ell$ that split completely in one extension 
                                   $\QQ\subset\QQ(\mu_{p^r},\sqrt[p^r]{\la_j})$, }}
{\displaystyle\hbox{$j=1$, $2$, but not in both or in I  of diagram \ref{Qext}}}
\right\}.
$$
Given $T>0$, let $\left(\sopra00\right)\notin{\cal V}_T\subset\ZZ^2$ be a finite set 
such that the restriction of the product of quotient maps
$$
{\cal V}_T\longrightarrow\prod_{\begin{subarray}{c}\ell\leq T\\ \ell\in{\cal P}_m^\flat\end{subarray}}
(\ZZ/\ell\ZZ)^2
$$
is a bijection and ${\cal V}_T\subseteq{\cal V}_{T^\prime}$ for $T\leq T^\prime$. 
Then, denoting (as usual) $\pi(T)$ the number of primes less than $T$ and 
making explicit the dependence of ${\cal P}_m$ on the initial vector,
$$
\delta_T:=\frac1{|{\cal V}_T|}\sum_{{\vec v}\in{\cal V}_T}
\frac
{\left|\left\{\hbox{$\ell\in{\cal P}({\vec v})_m$ such that $\ell\leq T$}\right\}\right|}
{\pi(T)}=\frac1{\pi(T)}
\sum_{\begin{subarray}{c}\ell\leq T\\ \ell\in{\cal P}_m^\flat\end{subarray}}\frac1\ell
$$
because $|E_\lambda|=\ell$. Thus, $\delta=\lim_{T\to\infty}\delta_T$ is the average density
of the sets ${\cal P}({\vec v})_m$ for ${\vec v}\in\bigcup_{T}{\cal V}_T$. On the other hand, 
$\delta_T<\pi(T)^{-1}\sum_{n=1}^T1/n$ and the well-known asymptotics 
$\pi(T)\sim T\log(T)^{-1}$ and $\sum_{n=1}^T1/n\sim\log(T)$ yield $\delta=0$. 
Since the set ${\cal V}_T$ can be constructed 
 so to contain any given $0\neq{\vec v}\in\ZZ^2$, we get an estimated density
$$
\delta({\cal P}(\vec v)_m)=0,\hbox{ for all ${\vec v}\neq0$,\qquad if $K=\QQ$}.
$$

Let us assume now that $K$ is 	quadratic and let $\la=\la_1$. Note that non-degeneracy 
is equivalent to the subgroup 
$\langle\la,\la^\tau\rangle<K^\times$ being free of rank 2. This time let $\Sigma_0$ 
be the finite set of primes containing $2$, the primes dividing ${\rm N}_{K/\QQ}(\la)$,
 the  primes such that $K\subset\QQ(\mu_{\ell^\infty})$ and the primes that 
 are ramified in $K$. Let $\ell\notin\Sigma_0$. If $\ell$ is split in K, then 
$\overline{\la}\in({\cal O}_K/\ell{\cal O}_K)^\times\simeq\FF_\ell^\times\times\FF_\ell^\times$.
The situation is very similar to the case $K=\QQ$ and we omit the details.

If $\ell$ is inert in $K$, then 
$\overline{\la}\in({\cal O}_K/\ell{\cal O}_K)^\times\simeq\FF_{\ell^2}^\times$. 
Write $\ell^2-1=ab$ where 
$\max\{\mathop{\rm Supp}(a)\}\leq m$ and $\min\{\mathop{\rm Supp}(b)\}> m$. Then 
\begin{eqnarray*}
	\max(\mathop{\rm Supp}({\rm ord}_\ell(\overline{M})))\leq m & \Longleftrightarrow &
	\mbox{$\overline{\la}$ is a $b$-power in $\FF_{\ell^2}^\times$}  \\
	 & \Longleftrightarrow & \sopra{\displaystyle\hbox{$\overline{\la}$ is a $p^r$-powers in 
	 $\FF_{\ell^2}^\times$ for all}}{\displaystyle\hbox{primes $p>m$ such that 
	 $p^r\mid\mid\ell^2-1$}}
\end{eqnarray*}
Let $\Sigma_{p,r}^\prime$ be the set of primes satisfying the latter condition at $p$.
Consider the diagram of Galois extensions
\begin{equation}
\label{Kext}
\xymatrix{
 & & K(\mu_{p^r},\sqrt[p^r]{\la}, \sqrt[p^r]{\la^\tau}) \\
 K(\mu_{p^r}) \ar@{-}[urr]^{\Ga_K^\prime} & & \\
\QQ(\mu_{p^r})\ar@{-}[u] & K \ar@{-}[ul]_H \ar@{-}[uur]^{\Ga_K}_{\rm III}  & \\
  & & \QQ \ar@{-}@/^/[ull] \ar@{-}[ul] \ar@{-}[uuu]_\Ga 
}.
\end{equation}
Then
$$
\Sigma_{p,r}^\prime=\{\text{primes $\ell$ that split completely in $\rm III$ and 
such that $\ell\not\equiv\pm1\bmod p^{r+1}$}\}.
$$
Again, $\Sigma_{p,r}^\prime\cap\Sigma_{p,r^\prime}^\prime=\emptyset$ if $r\neq r^\prime$
and if we let $\Sigma_p^\prime=\cup_{r\geq1}\Sigma_{p,r}^\prime$, then
\begin{equation}
\label{inert}
\widetilde{{\cal P}}_m^\prime=\{\text{$\ell\in{\cal P}_m^\prime$ such that 
$\ell$ is inert in $K$}\}=\bigcap_{p>m}\Sigma_p^\prime.
\end{equation}
The analogous of proposition \ref{propQ} is the following
\begin{proposition}\label{propK}
Suppose $p\notin\Sigma_0$ and $\la/\la^\tau$ not a root of $1$. Then:
\begin{enumerate}
  \item $[K(\mu_{p^r},\sqrt[p^r]{\la}):K(\mu_{p^r})]=
              [K(\mu_{p^r},\sqrt[p^r]{\la^\tau}):K(\mu_{p^r})]=p^r$;
  \item $K(\mu_{p^r},\sqrt[p^r]{\la})\cap K(\mu_{p^r},\sqrt[p^r]{\la^\tau})=K(\mu_{p^r})$;
  \item $\Gal(K(\mu_{p^r},\sqrt[p^r]{\la},\sqrt[p^r]{\tau(\la}))/K(\mu_{p^r}))\simeq
            (\ZZ/p^r\ZZ)^2$;
  \item $\QQ(\mu_{p^{r+1}})\cap K(\mu_{p^r},\sqrt[p^r]{\la},\sqrt[p^r]{\la^\tau})=\QQ(\mu_{p^r})$.
\end{enumerate}
\end{proposition}
To estimate the density of the primes in $\Sigma_{p,r}^\prime$, observe that an inert prime 
$\ell$ splits completely in the extension $({\rm I})$ of diagram \eqref{Kext} if and only 
if a Frobenius element 
$\sigma\in\mathop{\rm Frob}_{K(\mu_{p^r},\sqrt[p^r]{\la},\sqrt[p^r]{\la^\tau})/K}(\ell)\subset\Ga$
satisfies the following conditions:
$$
\sigma^2={\rm id}\qquad\text{and}\qquad\sigma_{|K}=\tau.
$$
These conditions define a conjugacy class $C\subset\Ga$ and by \v{C}ebotarev's
theorem we need to estimate its size. The exact sequences of Galois groups
$$
1\longrightarrow\Ga_K\longrightarrow\Ga\longrightarrow\langle\tau\rangle\longrightarrow1
$$
and
$$
1\longrightarrow\Ga_K^\prime\longrightarrow\Ga_K\longrightarrow H
\longrightarrow1
$$
split, so that 
$\Ga\simeq\Ga_K\ltimes\langle\tau\rangle\simeq(\Ga_K^\prime\ltimes H)\ltimes\langle\tau\rangle$. 
The extension $\QQ\subset K(\mu_{p^r})$ is abelian with Galois group isomorphic to 
$G=H\times\langle\tau\rangle$ so that we get $\Ga\simeq\Ga_K^\prime\ltimes G$.
Since $H$ is cyclic (of even order $p^{r-1}(p-1)$) there are $2$ elements of 
order 2 in $G$ restricting to $\tau$ and finally
$$
|C|\leq2|\Ga_K^\prime|=2p^{2r}.
$$
Combining this estimate with Dirichlet's theorem of primes in arithmetic
progression under the independence assumptions we get 
$$
\delta(\Sigma_{p,r}^\prime)\leq\left(\frac{p-1}p\right)\frac{2p^{2r}}{2p^{3r-1}(p-1)}=\frac1{p^r}.
$$
Thus, $\delta(\Sigma_p^\prime)\leq\sum_{r\geq1}p^{-r}=1/(p-1)$ and finally, 
from \eqref{inert} and recalling that the inert primes have density $1/2$,
$$
\delta(\widetilde{{\cal P}}_m^\prime)=\frac12
\prod_{\begin{subarray}{c}p>m\\ p\in\Sigma_0\end{subarray}}\delta(\Sigma_p^\prime)
\prod_{\begin{subarray}{c}p>m\\ p\notin\Sigma_0\end{subarray}}\frac1{p-1}=0.
$$

The analysis of the set ${\cal P}_m-{\cal P}_m^\prime$ follows the same lines of the 
$K=\QQ$ situation in the case of a split prime $\ell$ and we, again, omit the details. 
When $\ell$ is inert the basically trivial observation that $\lambda$ is a $b$-power if and only if 
$\overline\lambda$ is a $b$-power implies at once that
$$
\pi_\ell={\rm ord}_\ell(\overline{M})\qquad
\hbox{if $\ell$ is inert.}
$$
In other words, the set ${\cal P}_m-{\cal P}_m^\prime$ consists only of split primes or primes in $\Sigma_0$ and the heuristic estimate
$$
\delta({\cal P}(\vec v)_m)=0,\hbox{ for all ${\vec v}\neq0$,\qquad if $K$ is quadratic,}
$$
follows.

\section{Tables}
We implemented the procedure using the Maple 12 package and let it run on a MacBook. 
The tables in this section report some of these computations, done with a cutoff value 
$C_{\rm off}=10000$.

We consider all sequences up to shift-equivalence with positive parameters $A$ and $B$ 
such that $A+B\leq 4$ and non-negative initial values $G_0$ and $G_1$ such that 
$\max\{G_0,G_1\}\leq9$.

Tables 1--6 give the results of applying the procedure in search of pure powers for prime 
exponents $q$ with $3\leq q\leq 17$. Each table shows at the beginning the values $N_q$ 
which depend only on $A$, $B$ and the cutoff value. The tables contain 4 types of entries:
\begin{enumerate}
  \item an entry {\ok} indicates that the procedure outputs the empty set, i.e. that the 
            corresponding sequence does not contain $q$-th powers;
  \item an entry $\{a\}$ indicates that the procedure shows that the only $q$-th powers in the 
            corresponding sequence $\{G_n\}$ can occur only for $n\equiv a\bmod N_q$;
  \item an entry {\te am} indicates that the procedure shows that the only $q$-th powers in the 
            corresponding sequence $\{G_n\}$ can occur only for $n\equiv a\bmod (N_q/m)$;
  \item an entry $m$ indicates that the procedure final output was a set of $m$ different possible 
            classes modulo $N_q$ for indices $n$ with $G_n$ a $q$-th power, not coming from 
            the same class modulo $N_q/m$.
\end{enumerate}

Tables 7--12 list the $q$-power free values $2\leq k\leq 30$ for which the procedure shows 
that the equation \eqref{mainprob} has no solutions. 

\begin{center}
\tablefirsthead{%
      \multicolumn{8}{c}{TABLE 1} \\ 
      \multicolumn{8}{c}{$q$-powers in sequences with $A=1$ and $B=1$} \\
      \multicolumn{8}{c}{} \\ 
      \multicolumn{8}{c}{$N_3=186624$, $N_5=15552000$, $N_7=127008000$,} \\
      \multicolumn{8}{c}{$N_{11}=3841992000$, $N_{13}=43286443200$} \\ 
       \multicolumn{8}{c}{$N_{17}=68235175008000$} \\ 
     \hline
      $G_0$ & $G_1$ & $q=3$ & $q=5$ & $q=7$ & $q=11$ & $q=13$ & $q=17$ \\ \hline\hline}
\tablehead{%
\hline
\multicolumn{8}{|l|}{Table 1: $q$-powers in sequences with $A=1$, $B=1$}\\
\multicolumn{8}{|l|}{\small\sl (continued from previous page)}\\
      \hline
      $G_0$ & $G_1$ & $q=3$ & $q=5$ & $q=7$ & $q=11$ & $q=13$ & q=17 \\ \hline\hline}
\tabletail{%
      \hline
      \multicolumn{8}{|r|}{\small\sl continued on next page} \\
      \hline}
\tablelasttail{%
      \hline}
\begin{supertabular}{| c | c || c | c | c | c | c |c |}
0 & 1 & 96 & 42 & 18 &  14 & 20 & 26 \\ \hline
0 & 2 & 24 & 6 & \te02 & 6 & 18 & 6\\  \hline
0 & 3 & \te06 & \te04 & 6 & \te04 & \te02 & \te02 \\ \hline
0 & 4 & 24 & 14 & \te02 & \te06 & 6 & 18\\ \hline
0 & 5 & \te04 & \te04 &  \te04 &  6 & \te02 & 18 \\ \hline
0 & 6 & \te06 & \te02 & 12 & \te02 & \te02 &\te02 \\ \hline
0 & 7 & 18 & 6 & \te04 & 6 & \te02 & \te02 \\ \hline
0 & 8 & 96  & \te04 & \te02 & \te02 & \te02 & 10 \\  \hline
0 & 9 & 22  & 6  & 6 & 6 & \te02 & \te02 \\ \hline
1 & 3 & 48 & 8 & 8 & 16 & 8 & 32 \\ \hline
1 & 4 & 62 & 4 & 4 & \te02 & 4 & 4 \\ \hline
1 & 5 & 30 & 4 & \te02 & 4 & 4 & 12\\ \hline
1 & 6 & 4 & \te02 & \te04 & 8 & 4 & 6 \\ \hline
1 & 7 & 20 & \te02 & \te02 & 48 & \te02 & 12\\ \hline
1 & 8 & 64 & 4 & \te02 & \te02 & 4 & 32\\ \hline
1 & 9 & 32 & \te02 & \te02 & 8 & 4 & 48 \\ \hline
2 & 5 & 62 & 4 & 4 & \te{-2}2 & 4 & 4\\ \hline
2 & 6 & 24 & \ok & \ok & \ok & \ok & \ok \\ \hline
2 & 7 & \te{-3}2 & \ok  & \te{-9}2 & \ok & \ok & \ok \\ \hline
2 & 8 & \te{1}2 & \ok & \ok &  \ok & \ok & \ok \\ \hline
2 & 9 & \ok & \ok & \ok & \ok & \ok & \ok \\ \hline
3 & 7 & 30 & 4 & \te{-2}2 & 4 & 4 & 12\\ \hline
3 & 8 & \te12 & \ok & \te72  & \ok & \ok & \ok\\ \hline
3 & 9 & \ok  & \ok & \ok & \ok  & \ok  & \ok\\ \hline
4 & 9 & 4 & \te{-2}2 & \te{-2}4 & 8 & 4 & 6\\ \hline
6 & 4 & \te{-2}2 & \ok & \ok  & \ok & \ok & \ok\\ \hline
6 & 5 & 20 & \te{-1}2 & \te{-1}2 & 48 & \te{-1}2 & 12 \\ \hline
7 & 3 & \ok & \ok  & \ok  & \ok  & \ok & \ok \\ \hline
7 & 4 & \ok & \ok  & \ok  & \ok  & \ok & \ok  \\ \hline
7 & 5 & \ok &  \ok & \ok  & \ok  & \ok & \ok  \\ \hline
8 & 2 & 68 & \ok & \ok & \ok & \ok & \ok \\ \hline
8 & 3 & 52 & \ok & \ok & \ok & \ok & \ok  \\ \hline
8 & 4 & 40 & \ok & \ok & \ok & \ok & \ok \\ \hline
8 & 5 & 52 & \ok & \ok & \ok & \ok & \ok \\ \hline
8 & 6 & 68 & \ok & \ok & \ok & \ok & \ok \\ \hline
8 & 7 & 32 & \te{-1}2  & \te{-1}2 & 8 & 4 & 48 \\ \hline
9 & 1 & 44 & 16 & \te12 & 18 & 4 & 4 \\ \hline
9 & 2 & \ok & \ok & \ok & \ok & \ok & \ok \\ \hline
9 & 3 & 16 & 4 & \ok & \ok & \ok  & \ok \\ \hline
9 & 4 & \ok & \ok & \ok & \ok & \ok & \ok \\ \hline
9 & 5 & \ok & \ok & \ok & \ok & \ok & \ok \\ \hline
9 & 6 & 16 & 4 & \ok & \ok & \ok & \ok \\ \hline
9 & 7 & \ok & \ok & \ok & \ok & \ok & \ok \\ \hline
9 & 8 & 44 & 16 & \te{-1}2 & 18 & 4 & 4 \\ \hline
\end{supertabular}
\end{center}

\bigskip
\begin{center}
\tablefirsthead{%
      \multicolumn{8}{c}{TABLE 2} \\ 
      \multicolumn{8}{c}{$q$-powers in sequences with $A=1$ and $B=2$} \\
      \multicolumn{8}{c}{} \\ 
      \multicolumn{8}{c}{$N_3=31104$, $N_5=7776000$, $N_7=111132000$,} \\
      \multicolumn{8}{c}{$N_{11}=295833384000$, $N_{13}=86572886400$} \\ 
      \multicolumn{8}{c}{$N_{17}=393664471200$} \\ 
      \hline
      $G_0$ & $G_1$ & $q=3$ & $q=5$ & $q=7$ & $q=11$ & $q=13$ & $p=17$ \\ \hline\hline}
\tablehead{%
\hline
\multicolumn{8}{|l|}{Table 2: $q$-powers in sequences with $A=1$, $B=2$}\\
\multicolumn{8}{|l|}{\small\sl (continued from previous page)}\\
      \hline
      $G_0$ & $G_1$ & $q=3$ & $q=5$ & $q=7$ & $q=11$ & $q=13$ & $q=17$ \\ \hline\hline}
\tabletail{%
      \hline
      \multicolumn{8}{|r|}{\small\sl continued on next page} \\
      \hline}
\tablelasttail{%
      \hline}
\begin{supertabular}{| c | c || c | c | c | c | c | c |}
0 & 1 & 4 & 8 & 5 & 15 & 6 & 3 \\ \hline
0 & 2 & 2 & 2 & 4 & 17 & 4 & 3\\ \hline
0 & 3 & \{0\} & \{0\} & \{0\} & 15 & \{0\} & \{0\} \\ \hline
0 & 4 & 2 & 6 & 4 & 5 & 13 & 2\\ \hline
0 & 5 & \te02 & \{0\} & 3 & 3 & \{0\} & \{0\} \\ \hline
0 & 6 & 3 & 5 & \{0\} & 3 & \{0\} & 9\\ \hline
0 & 7 & \te{0}2 & \{0\} & \{0\} & 3 & \{0\} & 3 \\ \hline
0 & 8 & 3 & \{0\} & 3 & 3 & 3 & 3\\ \hline
0 & 9 & 5 & \{0\} & \{0\} & \{0\} & \{0\} & \{0\} \\ \hline
1 & 4 & 4 & \{0\} & 2 & 2 & 3 & \{0\}\\ \hline
1 & 5 & \{0\} & \{0\} & 3 & 8 & \{0\} & 4\\ \hline
1 & 6 & 2 & 3 & 2 & 3 & 6 & 8\\ \hline
1 & 7 & 5 & 3 & 6 & 4 & 6 & 4\\ \hline
1 & 8 & 6 & 3 & \{0\} & 4 & 2 & 6 \\ \hline
1 & 9 & \{0\} & \{0\} & 4 & 4 & \{0\} & 4 \\ \hline
2 & 3 & 2 & \ok & \ok & \ok & \ok & \ok \\ \hline
2 & 5 & \ok & \ok & \ok & \ok & \ok & \ok \\ \hline
2 & 7 & \ok & \ok & \ok & \ok & \ok & \ok \\ \hline
2 & 8 & \te12 & \ok & \ok & \ok & \ok & \ok \\ \hline
2 & 9 & \ok & \ok & \ok & \ok  & \ok & \ok \\ \hline
3 & 4 & \ok & \ok & \ok & \ok & \ok & \ok \\ \hline
3 & 6 & \ok & \ok & \ok & \ok & \ok & \ok \\ \hline
3 & 8 & \{1\} & \ok & \ok & \ok & \ok  & \ok \\ \hline
4 & 2 & \{-1\} & \{-1\} & 3 & 8 &  \{-1\} & 4 \\ \hline
4 & 3 & \ok & \ok  & \ok & \ok & \ok & \ok \\ \hline
4 & 5 & \ok & \ok & \ok  & \ok & \ok & \ok \\ \hline
4 & 7 & \{-3\} & \ok & \ok & \ok & \ok & \ok \\ \hline
4 & 9 & \ok & \ok & \ok & \ok & \ok & \ok \\ \hline
5 & 0 & \{1\} & \{1\} & 3 & 9 & \{1\} & \{1\} \\ \hline
5 & 1 & \te12 & \{1\} & 4 & 2 & 2 & \{1\} \\ \hline
5 & 2 & \ok  & \ok & \ok & \ok & \ok & \ok \\ \hline
5 & 3 & \{-1\} & 3 & 2 & 3 & 6 & 8 \\ \hline
5 & 4 & \ok & \ok & \ok & \ok & \ok & \ok \\ \hline
5 & 6 & \ok & \ok & \ok & \ok & \ok & \ok \\ \hline
5 & 8 & 4 & \ok & \ok & \ok & \ok & \ok \\ \hline
6 & 0  & \te12 & 3 & 3 & 3 & \te13 & \te12 \\ \hline
6 & 1 & 2 & \{1\} & \te13 & \{1\} & \{1\} & 6 \\ \hline
6 & 2 & \{-6\} & \ok & \ok & \ok & \ok & \ok \\ \hline
6 & 3 & \ok & \ok & \ok & \ok & \ok & \ok \\ \hline
6 & 4 & 2 & 2 & 4 & 2 & 3 & 3\\ \hline
6 & 5 & \{3\} & \ok & \ok & \ok & \ok & \ok \\ \hline
6 & 7 & \ok & \ok & \ok & \ok & \ok & \ok \\ \hline
6 & 9 & \ok & \ok & \ok & \ok & \ok & \ok \\ \hline
7 & 0 & \{1\} & \te12 & 3 & \{1\} & \{1\} & \{1\} \\ \hline
7 & 1 & \te12 & \{1\} & \{1\} & 2 & 6 & 2\\ \hline
7 & 2 & \ok & \ok & \ok & \ok & \ok & \ok \\ \hline
7 & 3 & \ok & \ok & \ok & \ok & \ok & \ok \\ \hline
7 & 4 & \ok & \ok & \ok & \ok & \ok & \ok \\ \hline
7 & 5 & 3 & 3 & \{-1\} & 4 & 2 & 6\\ \hline
7 & 6 & \ok & \{3\} & \ok & \ok & \ok & \ok \\ \hline
7 & 8 & \te12 & \ok & \ok & \ok & \ok & \ok \\ \hline
8 & 0 & 2 & \{1\} & 9 & 27 & 3 & \{1\} \\ \hline
8 & 1 & 4 & \{1\} & \{1\} & 2 & 2 & \{1\} \\ \hline
8 & 2 & \{0\} & \ok  & \ok & \ok & \ok & \ok \\ \hline
8 & 3 & 2 & \ok & \ok & \ok & \ok & \ok \\ \hline
8 & 4 & 2 & \ok & \ok & \ok & \ok & \ok \\ \hline
8 & 5 & 2 & \ok & \ok & \ok & \ok & \ok \\ \hline
8 & 6 & 2 & \{-1\} & 4 & 4 & \{-1\} & 4 \\ \hline
8 & 7 & 2 & \ok & \ok & \ok & \ok & \ok \\ \hline
9 & 0 & \te12  & \te12 &  \{1\} & \{1\} & \{1\} & \{1\} \\ \hline
9 & 1 & \{1\} & 2 & 2 & \{1\} & 2 & 2 \\ \hline
9 & 2 & \{4\} & \ok & \ok & \ok & \ok & \ok \\ \hline
9 & 3 & \{3\} & \ok & \ok & \ok & \ok & \ok \\ \hline
9 & 4 & \ok & \ok & \ok & \ok & \ok & \ok \\ \hline
9 & 5 & \ok & \ok & \ok & \ok & \ok & \ok \\ \hline
9 & 6 & \{-3\} & \ok & \ok & \ok & \ok & \ok \\ \hline
9 & 7 & \{-1\} & 2 & 2 & 6 & 2 & 2\\ \hline
9 & 8 & \te12 & \ok & \ok & \ok & \ok & \ok \\ \hline
\end{supertabular}
\end{center}

\bigskip
\begin{center}
\tablefirsthead{%
      \multicolumn{8}{c}{TABLE 3} \\ 
      \multicolumn{8}{c}{$q$-powers in sequences with $A=2$ and $B=1$} \\
      \multicolumn{8}{c}{} \\ 
      \multicolumn{8}{c}{$N_3=41472$, $N_5=15552000$, $N_7=74088000$} \\
      \multicolumn{8}{c}{$N_{11}=12074832000$, $N_{13}=519437318400$} \\ 
      \multicolumn{8}{c}{$N_{17}=787328942400$} \\ 
     \hline
      $G_0$ & $G_1$ & $q=3$ & $q=5$ & $q=7$ & $q=11$ & $q=13$ & $q=17$ \\ \hline\hline}
\tablehead{%
\hline
\multicolumn{8}{|l|}{Table 3: $p$-powers in sequences with $A=2$, $B=1$}\\
\multicolumn{8}{|l|}{\small\sl (continued from previous page)}\\
      \hline
      $G_0$ & $G_1$ & $q=3$ & $q=5$ & $q=7$ & $q=11$ & $q=13$ & $q=17$ \\ \hline\hline}
      \tabletail{%
      \hline
      \multicolumn{8}{|r|}{\small\sl continued on next page} \\
      \hline}
\tablelasttail{%
      \hline}
\begin{supertabular}{| c | c || c | c | c | c | c | c |}
0 & 1 & 6 & 26 & 26 & 14 & 38 & 14\\ \hline
0 & 2 & \te02 & 18 & \te02 & 6 & 6 & 6\\ \hline
0 & 3 & \te06 & \te04 & 10 & 6 & 6 & 10 \\ \hline
0 & 4 & 12 & 6 & \te02 & 6 & 6 & 6\\ \hline
0 & 5 & \te02 & 12 & 6 & \te02 & \te02 & 18\\ \hline
0 & 6 & 6 & \te02 & \te02 & \te02 & \te02 & 6\\ \hline
0 & 7 & \te02 & 6 & \te02 & 18 & \te02& \te02 \\ \hline
0 & 8 & 6 & 24 & \te06 & 6 & 6 & 6\\ \hline
0 & 9 & \te02 & 6 & 10 & 6 & \te02 & \te02\\ \hline
1 & 1 & 8 & 44 & 16 & 8 & 84 & 8\\ \hline
1 & 2 & 6 & 26 & 26 & 14 & 38 & 14\\ \hline
1 & 3 & 8 & 44 & 16 & 8 & 84 & 8\\ \hline
1 & 4 & 8 & 6 & 6 & 6 & \te02 & 4 \\ \hline
1 & 5 & 12 & \te02 & 4 & 4 & 4 & \te02\\ \hline
1 & 6 & \te02 & 10 & 6 & 12 & 6 & 4\\ \hline
1 & 7 & 4 & 8 & \te02 & 4 & 12 & 8\\ \hline
1 & 8 & 6 & 4 & 12 & 8 & 4 & 8\\ \hline
1 & 9 & \te02 & 8 & \te02 & 6 & 4 & \te02 \\ \hline
2 & 2 & \ok & \ok & \ok & \ok & \ok & \ok \\ \hline
2 & 3 & 8 & 6 & 6 & 6 & \te{-1}2 & 4 \\ \hline
2 & 7 & \ok & \ok & \ok & \ok & \ok & \ok \\ \hline
2 & 8 & \te12 &\ok & \ok  & \ok & \ok & \ok \\ \hline
2 & 9 & \te{-2}8 & \ok & \ok & \ok & \ok & \ok \\ \hline
3 & 3 & \ok & \ok & \ok & \ok & \ok & \ok \\ \hline
3 & 4 & \ok & \ok & \ok & \ok & \ok & \ok \\ \hline
3 & 5 & 12 & \te{-1}2 & 7 & 4 & 4 & \te{-1}2 \\ \hline
3 & 9 & \ok & \ok & \ok & \ok & \ok & \ok \\ \hline
4 & 3 & \ok & \ok & \ok & \ok & \ok & \ok \\ \hline
4 & 4 & \ok & \ok & \ok & \ok & \ok & \ok \\ \hline
4 & 5 & \ok & \ok & \ok & \ok & \ok & \ok \\ \hline
4 & 6 & \te{-2}2 & \ok & \ok & \ok & \ok & \ok \\ \hline
4 & 7 & \te{-1}2 & 10 & 6 & 12 & 6 & 4\\ \hline
5 & 3 & \ok & \ok &  \ok &  \ok  &  \ok & \ok  \\ \hline
5 & 4 & \ok & \ok &  \ok  &  \ok  &  \ok  & \ok \\ \hline
5 & 5 & \ok & \ok &  \ok  &  \ok  &  \ok & \ok  \\ \hline
5 & 6 & \ok & \ok &  \ok  &  \ok  &  \ok & \ok  \\ \hline
5 & 7 & \ok & \ok  & \ok &  \ok  &  \ok & \ok \\ \hline
5 & 8 & \te18 & \ok & \ok &  \ok  & \ok & \ok  \\ \hline
5 & 9 & 4 & 4 & 8 & \te{-1}2 & 4 & 8 \\ \hline
6 & 2 & \ok & \ok & \ok & \ok & \ok  & \ok \\ \hline
6 & 3 & 8 & \ok & \ok & \ok & \ok & \ok \\ \hline
6 & 4 & \te{-1}4 & \te32 & \ok & \ok  & \ok & \ok \\ \hline
6 & 5 & \ok & \ok & \ok & \ok & \ok & \ok \\ \hline
6 & 6 & \ok & \ok & \ok & \ok & \ok & \ok \\ \hline
6 & 7 & \ok & \ok & \ok & \ok & \ok & \ok \\ \hline
6 & 8 & \te14 & \te{-3}2 & \ok & \ok & \ok & \ok \\ \hline
6 & 9 & 8 & \ok & \ok & \ok & \ok & \ok \\ \hline
7 & 2 & \te62 & \ok & \ok & \ok & \ok & \ok \\ \hline
7 & 3 & \ok & \ok & \ok & \ok & \ok & \ok \\ \hline
7 & 4 & 8 & \ok & \ok & \ok & \ok & \ok \\ \hline
7 & 5 & \ok & \ok & \ok & \ok & \ok & \ok \\ \hline
7 & 6 & \te{-1}2 & \ok & \ok & \ok & \ok & \ok \\ \hline
7 & 7 & \ok & \ok & \ok & \ok & \ok & \ok \\ \hline
7 & 8 & \te12 & \ok & \ok & \ok & \ok & \ok \\ \hline
7 & 9 & \ok & \ok & \ok & \ok & \ok & \ok \\ \hline
8 & 1 & 12 & 4 & \te12 & 6 & \te12 & \te12 \\ \hline
8 & 2 & 10 & \ok & \ok & \ok & \ok & \ok \\ \hline
8 & 3 & 6 & \ok & \ok & \ok & \ok & \ok \\ \hline
8 & 4 & 16 & 4 & \ok & \ok & \ok & \ok \\ \hline
8 & 5 & \te02 & \ok & \ok & \ok & \ok & \ok \\ \hline
8 & 6 & \te02 & \ok & \ok & \ok & \ok & \ok \\ \hline
8 & 7 & \te02 & \ok & \ok & \ok & \ok & \ok \\ \hline
8 & 8 & 8 & \ok & \ok & \ok & \ok & \ok \\ \hline
8 & 9 & \te02 & \ok & \ok & \ok & \ok & \ok  \\ \hline
9 & 1 & 6 & \te12 & 16 & 4 & \te12 & 12\\ \hline
9 & 2 & \ok & \ok & \ok & \ok & \ok & \ok \\ \hline
9 & 3 & \ok & \ok & \ok & \ok & \ok & \ok \\ \hline
9 & 4 & \ok & \ok & \ok & \ok & \ok & \ok \\ \hline
9 & 5 & \ok & \ok & \ok & \ok & \ok & \ok \\ \hline
9 & 6 & \ok & \ok & \ok & \ok & \ok & \ok \\ \hline
9 & 7 & \ok & \ok & \ok & \ok & \ok & \ok \\ \hline
9 & 8 & \te12 & \ok & \ok & \ok & \ok & \ok \\ \hline
9 & 9 & 8 & \ok & \ok & \ok & \ok & \ok \\ \hline
\end{supertabular}
\end{center}

\bigskip
\begin{center}
\tablefirsthead{%
      \multicolumn{8}{c}{TABLE 4} \\ 
      \multicolumn{8}{c}{$q$-powers in sequences with $A=1$ and $B=3$} \\
      \multicolumn{8}{c}{} \\ 
      \multicolumn{8}{c}{$N_3=46656$, $N_5=3888000$, $N_7=296335200$} \\
      \multicolumn{8}{c}{$N_{11}=658627200$, $N_{13}=865728864000$} \\
      \multicolumn{8}{c}{$N_{17}=257297040000$} \\
     \hline
      $G_0$ & $G_1$ & $q=3$ & $q=5$ & $q=7$ & $q=11$ & $q=13$ & $q=17$ \\ \hline\hline}
      \tablehead{%
\hline
\multicolumn{8}{|l|}{Table 4: $q$-powers in sequences with $A=1$, $B=3$}\\
\multicolumn{8}{|l|}{\small\sl (continued from previous page)}\\
      \hline
      $G_0$ & $G_1$ & $q=3$ & $q=5$ & $q=7$ & $q=11$ & $q=13$ & $q=17$ \\ \hline\hline}
      \tabletail{%
      \hline
      \multicolumn{8}{|r|}{\small\sl continued on next page} \\
      \hline}
\tablelasttail{%
      \hline}
\begin{supertabular}{| c | c || c | c | c | c | c | c |}
0 & 1 & 105 & 4 & 20 & 8 & 21& 16 \\ \hline
0 & 2 & 69 & \{0\} & \{0\} & 3 & \{0\} & 6 \\ \hline
0 & 3 & 69 & 6 & 6 & 17 & 5 & 8 \\ \hline
0 & 4 & 21 & \{0\} & \{0\} & 3 & 3 & \te02\\ \hline
0 & 5 & 81 & 3 & 6 & 3 & \{0\} & \te02 \\ \hline
0 & 6 & 63 & 12 & \te02 & \{0\} & \{0\} & \te02\\ \hline
0 & 7 & 144 & 3 & \{0\} & 3 & \{0\} & \te02\\ \hline
0 & 8 & 105 & 5 & \te06 & 3 & 9 & 18 \\ \hline
0 & 9 & 63 & 2 & 9 & 12 & 17 & 16 \\ \hline
1 & 2 & 75 & 4 & 2 & 4 & 16 &\te02 \\ \hline
1 & 3 & 12 & \{0\} & 2 & 12 & 2 & \te04 \\ \hline
1 & 4 & 105 & 4 & 20 & 8 & 21 & 16 \\ \hline
1 & 5 & 51 & 2 & 4 & 16 & 3 & 4\\ \hline
1 & 6 & 75  & \te02 & 24 & 10 & 6 & 4\\ \hline
1 & 7 & 54 & \{0\} & \{0\} & 3 & 2 & \te02 \\ \hline
1 & 8 & 69 & \te02 & 4 & 4 & 2 & \te02 \\ \hline
1 & 9 & 39 & \{0\} & 2 & 8 & \{0\} & 4 \\ \hline
2 & 3 & 6 & \ok & \ok & \ok & \ok & \ok \\ \hline
2 & 4 & 15 & \ok  & \ok & \ok & \ok & \ok  \\ \hline
2 & 6 & 18 & \ok  & \ok & \ok & \ok & \ok \\ \hline
2 & 7 & 60 & \ok & \ok & \ok & \ok & \ok \\ \hline
2 & 9 & 48 & \ok & \ok & \ok & \ok & \ok \\ \hline
3 & 1 & 33 & \{1\} & 2 & 9 & 2 & \te14 \\ \hline
3 & 2 & \ok & \ok & \ok & \ok & \ok & \ok \\ \hline
3 & 4 & 93 & \ok & \ok & \ok & \ok& \ok  \\ \hline
3 & 5 & 30 & \ok & \ok & \ok & \ok & \ok \\ \hline
3 & 7 & 12 & \ok & \ok & \ok & \ok & \ok \\ \hline
3 & 8 & 27 & 3 & \ok & \ok & \ok & \ok \\ \hline
4 & 0 & \te13 & \te14 & 3 & 7 & \{1\} & \te12 \\ \hline
4 & 1 & 99 & 4 & 5 & 18 & 5 & 8\\ \hline 
4 & 2 & 66 & \ok & \ok & \ok & \ok & \ok \\ \hline
4 & 3 & 24 & \ok & \ok & \ok & \ok & \ok \\ \hline
4 & 5 & 30 & \{3\} & \ok & \ok & \ok & \ok \\ \hline
4 & 6 & 36 & \ok & \ok & \ok & \ok & \ok \\ \hline
4 & 8 & 30 & \ok & \ok & \ok & \ok & \ok \\ \hline
4 & 9 & 72 & \ok & \ok & \ok & \ok & \ok \\ \hline
5 & 0 & 42 & \te14 & 3 & \{1\} & \{1\} & \te12 \\ \hline
5 & 1 & 9 & 2 & 2 & 6 & 12 & 8 \\ \hline
5 & 2 & 87 & 7 & 28 & 26 & 24 & 10\\ \hline
5 & 3 & 12 & \ok & \ok & \ok & \ok & \ok \\ \hline
5 & 4 & 6 & \ok & \ok & \ok & \ok & \ok \\ \hline
5 & 6 & 39 & \ok & \ok & \ok & \ok & \ok  \\ \hline
5 & 7 & 63 & \ok & \ok & \ok & \ok & \ok \\ \hline
5 & 9 & 96 & \ok & \ok & \ok & \ok & \ok \\ \hline
6 & 0 & 108 & \{1\} & \{1\} & \{1\}  & 3 & \te12\\ \hline
6 & 1 & 33 & \{1\} & 8 & 10 & 6 & \te12\\ \hline
6 & 2 & \ok & \ok & \ok & \ok & \ok & \ok \\ \hline
6 & 3 & 33 & \{-1\} & \{-1\} & 3 & 2 & \te{-1}2 \\ \hline
6 & 4 & 36 & \ok & \ok & \ok & \ok & \ok \\ \hline
6 & 5 & 27 & \ok & \ok & \ok & \ok & \ok \\ \hline
6 & 7 & 27 & \ok & \ok & \ok & \ok & \ok  \\ \hline
6 & 8 & 9 & \ok & 4 & \ok & \ok & \ok  \\ \hline
7 & 0 & 42 & \{1\} & \{1\} & \{1\}  & \{1\} & 6 \\ \hline
7 & 1 & 42 & \{1\} & 2 & 18 & 4 & 4 \\ \hline
7 & 2 & 3 & \ok & \ok & \ok & \ok & \ok \\ \hline
7 & 3 & 12 & \ok & \ok & \ok & \ok & \ok \\ \hline
7 & 4 & 18 & \te{-1}2 & 4 & 4 & 2 & \te{-1}2 \\ \hline
7 & 5 & 18 & \ok & \ok & \ok & \ok & \ok \\ \hline
7 & 6 & 33 & \ok & \ok & \ok & \ok & \ok \\ \hline
7 & 8 & 57 & \ok & \ok & \ok & \ok & \ok  \\ \hline
7 & 9 & 33 & \ok & \ok & \ok & \ok & \ok \\ \hline
8 & 0 & 69 & \te14 & \te13 & 5 & \{1\} & \te12 \\ \hline
8 & 1 & 60 & \{1\} & 5 & 9 & 13 & 4 \\ \hline
8 & 2 & 48 & \{3\} & \ok & \ok & \ok & \ok \\ \hline
8 & 3 & 54 & \ok & \ok & \ok & \ok & \ok \\ \hline
8 & 4 & 177 & \ok & \ok & \ok & \ok & \ok \\ \hline
8 & 5 & 48 & \{-1\} & 2 & 8 & \{-1\} & 4 \\ \hline
8 & 6 & 36 & \ok & \ok & \ok & \ok & \ok \\ \hline
8 & 7 & 45 & \ok & \ok & \ok & \ok& \ok  \\ \hline
8 & 9 & 24 & \ok & \ok & \ok & \ok & \ok \\ \hline
9 & 0 & 105 & \{1\} & 5 & 9 & \{1\} & 6 \\ \hline
9 & 1 & 36 & 6 & \{1\} & 2 & 2 & 8\\ \hline
9 & 2 & 36 & \ok & \ok & \ok & \ok & \ok  \\ \hline
9 & 3 & \ok & \ok & \ok & \ok & \ok & \ok \\ \hline
9 & 4 & 36 & \ok & \ok & \ok & \ok & \ok \\ \hline
9 & 5 & 12 & \{2\} & \ok & \ok & \ok & \ok \\ \hline
9 & 6 & 39 & \{-1\} & 2 & 9 & 2 & \te{-1}4 \\ \hline
9 & 7 & 63 & \ok & \ok & \ok & \ok & \ok \\ \hline
9 & 8 & 30 & \ok & \ok & \ok & \ok & \ok \\ \hline
\end{supertabular}
\end{center}

\bigskip
\begin{center}
\tablefirsthead{%
      \multicolumn{8}{c}{TABLE 5} \\ 
      \multicolumn{8}{c}{$q$-powers in sequences with $A=2$ and $B=2$} \\
      \multicolumn{8}{c}{} \\ 
      \multicolumn{8}{c}{$N_3=62208$, $N_5=7776000$, $N_7=177811200$} \\
      \multicolumn{8}{c}{$N_{11}=59166676800$, $N_{13}=719566848000$} \\
      \multicolumn{8}{c}{$N_{17}=4374049680000$} \\
      \hline
      $G_0$ & $G_1$ & $q=3$ & $q=5$ & $q=7$ & $q=11$ & $q=13$ & $q=17$ \\ \hline\hline}
      \tablehead{%
\hline
\multicolumn{8}{|l|}{Table 5: $q$-powers in sequences with $A=2$, $B=2$}\\
\multicolumn{8}{|l|}{\small\sl (continued from previous page)}\\
      \hline
      $G_0$ & $G_1$ & $q=3$ & $q=5$ & $q=7$ & $q=11$ & $q=13$ & $q=17$ \\ \hline\hline}
      \tabletail{%
      \hline
      \multicolumn{8}{|r|}{\small\sl continued on next page} \\
      \hline}
\tablelasttail{%
      \hline}
\begin{supertabular}{| c | c || c | c | c | c | c | c |}
0 & 1 & 32 & 6 & 16 & 18 & 214 & 24\\ \hline
0 & 2 & 28 & 33 & 24 & 60 & 300 & 36 \\ \hline
0 & 3 & \te04 & \{0\} & \te06 & 18 & 54 & 10\\ \hline
0 & 4 & 24 & 5 & \te04 & 54 & 54 & 50\\ \hline
0 & 5 & \te04 & \{0\} & \te02 & 18 & 54 & 6\\ \hline
0 & 6 & \te04 & \te02 & 6 & 18 & 94 & 6 \\ \hline
0 & 7 & \te04 & 5 & \te06 & 86 & 22 & 6 \\ \hline
0 & 8 & 32 & \{0\} & 22 & 18 & 230 & \te02\\ \hline
0 & 9 & \te08 & 5 & 6 & 6 & 26 & 6 \\ \hline
1 & 1 & 32 & 8 & 26 & 60 & 198 & 56 \\ \hline
1 & 3 & 24 & 5 & 10 & 54 & 196 & 30 \\ \hline
1 & 5 & 8 & 6 & 16 & 24 & 148 & 18\\ \hline
1 & 6 & 8 & \{0\} & 8 & 18 & 56 & 8\\ \hline
1 & 7 & 8 & 3 & 36 & 6 & 288 & 12 \\ \hline
1 & 8 & 8 & 2 & 6 & 18 & 106 & 20 \\ \hline
1 & 9 & 8 & \{0\} & \te02 & 12 & 60 & 16 \\ \hline
2 & 2 & 8 & \{-1\} & \te{-1}2 & 12 & 56 & 4\\ \hline
2 & 3 & \te{-5}4 & \ok & \ok & \ok & \ok & \ok \\ \hline
2 & 5 & \ok & \ok & \ok & \ok & \ok & \ok \\ \hline
2 & 7 & \ok & \ok & \ok & \ok & \ok & \ok \\ \hline
2 & 8 & 8 & \{-2\} & \te{-2}2 & 12 & 56 & 4 \\ \hline
2 & 9 & \ok & \ok & \ok & \ok & \ok & \ok \\ \hline
3 & 2 & \ok & \ok & \ok & \ok & \ok & \ok \\ \hline
3 & 3 & \ok & \ok & \ok & \ok & \ok & \ok \\ \hline
3 & 4 & \te{-1}4 & 4 & 8 & 12 & 72 & 12 \\ \hline
3 & 5 & \ok & \ok & \ok & \ok & \ok & \ok \\ \hline
3 & 7 & \te{-2}4 & 3 & 4 & 36 & 164 & 4 \\ \hline
3 & 9 & \te{-2}4 & 3 & \te{-2}2 & 18 & 44 & 10 \\ \hline
4 & 1 & 16 & \{1\} & 8 & 36 & 128 & 4 \\ \hline
4 & 2 & \te52 & \ok & \ok & \ok & \ok & \ok \\ \hline
4 & 3 & 8 & \ok & \ok & \ok & \ok & \ok \\ \hline
4 & 4 & \ok & \ok & \ok & \ok & \ok & \ok \\ \hline
4 & 5 & \ok & \ok & 8 & \ok & \ok & \ok \\ \hline
4 & 6 & 8 & \{-1\} & 8 & 18 & 56 & 8 \\ \hline
4 & 7 & \ok & \ok & \ok & \ok & \ok & \ok \\ \hline
4 & 9 & \ok & \ok & \ok & \ok & \ok & \ok \\ \hline
5 & 0 & \te14 & \te12 & 6 & 18 & 54& 10 \\ \hline
5 & 1 & 8 & 3 & 4 & 24 & 76 & 80\\ \hline
5 & 2 & \ok & 3 & \ok & \ok & \ok & \ok \\ \hline
5 & 3 & \ok & \ok & \ok & \ok & \ok & \ok \\ \hline
5 & 4 & \ok & \ok & \ok & \ok & \ok & \ok \\ \hline
5 & 5 & \ok & \ok & \ok & \ok & \ok & \ok \\ \hline
5 & 6 & \ok & \ok & \ok & \ok & \ok & \ok \\ \hline
5 & 7 & \ok & \ok & \ok & \ok & \ok & \ok \\ \hline
5 & 8 & 16 & 3 & 36 & 6 & 288 & 12 \\ \hline
5 & 9 & \ok & \ok & \ok & \ok & \ok & \ok \\ \hline
6 & 0 & \te14 & \{1\} & 6 & 6 & 78 & 6 \\ \hline
6 & 1 & 16 & 3 & 6 & 24 & 44 & 6 \\ \hline
6 & 2 & 16 & \ok & \ok & \ok & \ok & \ok \\ \hline
6 & 3 & \ok & \ok & \ok & \ok & \ok & \ok \\ \hline
6 & 4 & \ok & \ok & \ok & \ok & \ok & \ok \\ \hline
6 & 5 & 8 & \ok & \ok & \ok & \ok & \ok \\ \hline
6 & 6 & \ok & \ok & \ok & \ok & \ok & \ok \\ \hline
6 & 7 & \ok & \ok & \ok & \ok & \ok & \ok \\ \hline
6 & 8 & \te14 & \ok & \ok & \ok & \ok & \ok \\ \hline
6 & 9 & 8 & \ok & \ok & \ok & \ok & \ok \\ \hline
7 & 0 & \te14 & 3 & 6 & 18 & 138 & 6 \\ \hline
7 & 1 & \te14 & 3 & 12 & 36 & 144 & 40\\ \hline
7 & 2 & \ok & \ok & \ok & \ok & \ok & \ok \\ \hline
7 & 3 & \ok & \ok & \ok & \ok & \ok & \ok \\ \hline
7 & 4 & \ok & \ok & \ok & \ok & \ok & \ok \\ \hline
7 & 5 & te{-2}4 & \ok & \ok & \ok & \ok & \ok \\ \hline
7 & 6 & 8 & \ok & \ok & \ok & \ok & \ok \\ \hline
7 & 7 & \ok & \ok & \ok & \ok & \ok & \ok \\ \hline
7 & 8 & 16 & \ok & \ok & \ok & \ok & \ok \\ \hline
7 & 9 & \ok & \{2\} & \ok & \ok & \ok & \ok \\ \hline
8 & 0 & 28 & 5 & \te12 & 54 & 26 & 18 \\ \hline
8 & 1 & 20 & 2 & 6 & 12 & 58 & 8 \\ \hline
8 & 2 & 8 & \ok & 8 & \ok & \ok & \ok \\ \hline
8 & 3 & \te04 & \ok & \ok & \ok & \ok & \ok \\ \hline
8 & 4 & \te04 & \ok & \ok & \ok & \ok & \ok \\ \hline
8 & 5 & \te04 & \ok & \ok & \ok & \ok & \ok \\ \hline
8 & 6 & \te04 & \ok & \ok & \ok & \ok & \ok \\ \hline
8 & 7 & 12 & \ok & \ok & \ok & \ok & \ok \\ \hline
8 & 8 & 32 & 3 & \ok & \ok & \ok & \ok \\ \hline
8 & 9 & 12 & \ok & \ok & \ok & \ok & \ok \\ \hline
9 & 0 & \te14 & 3 & 12 & 54 & 46 & 10 \\ \hline
9 & 1 & \te1{12} & 10 & 12 & 12 & 80 & 64 \\ \hline
9 & 2 & \te{-1}4 & \ok & \ok & \ok & \ok & \ok \\ \hline
9 & 3 & \ok & \ok & \ok & \ok & \ok & \ok \\ \hline
9 & 4 & \ok & \ok & \ok & \ok & \ok & \ok \\ \hline
9 & 5 & \ok & \ok & \ok & \ok & \ok & \ok \\ \hline
9 & 6 & \ok & \ok & \ok & \ok & \ok & \ok \\ \hline
9 & 7 & \ok & \{2\} & \ok & \ok & \ok & \ok \\ \hline
9 & 8 & \te14 & \ok & \ok & \ok & \ok & \ok \\ \hline
9 & 9 & \ok & \ok & \ok & \ok & \ok & \ok \\ \hline
\end{supertabular}
\end{center}

\bigskip
\begin{center}
\tablefirsthead{%
      \multicolumn{8}{c}{TABLE 6} \\ 
      \multicolumn{8}{c}{$q$-powers in sequences with $A=3$ and $B=1$} \\
      \multicolumn{8}{c}{} \\ 
      \multicolumn{8}{c}{$N_3=93312$, $N_5=15552000$, $N_7=148176000$} \\
      \multicolumn{8}{c}{$N_{11}=46103904000$, $N_{13}=432864432000$} \\                     
      \multicolumn{8}{c}{$N_{17}=102918816000$}  \\
      \hline
      $G_0$ & $G_1$ & $q=3$ & $q=5$ & $q=7$ & $q=11$ & $q=13$ & $q=17$ \\ \hline\hline}
\tablehead{%
\hline
\multicolumn{8}{|l|}{Table 6: $q$-powers in sequences with $A=3$, $B=1$}\\
\multicolumn{8}{|l|}{\small\sl (continued from previous page)}\\
      \hline
      $G_0$ & $G_1$ & $q=3$ & $q=5$ & $q=7$ & $q=11$ & $q=13$ & $q=17$ \\ \hline\hline}
\tabletail{%
      \hline
      \multicolumn{8}{|r|}{\small\sl continued on next page} \\
      \hline}
\tablelasttail{%
      \hline}
\begin{supertabular}{| c | c || c | c | c | c | c | c |}
0 & 1 & 70 & 10 & 6 & 6 & 10 & 6 \\ \hline
0 & 2 & 22 & \te02 & \te02 & \te02 & \te02 & 6\\ \hline
0 & 3 & 22 & \te02 & 30 & 6 & \te02 & \te02 \\ \hline
0 & 4 & 44 & \te02 & \te02 & 36 & \te02 & \te02 \\ \hline
0 & 5 & 22 & \te02 & 36 & \te02 & \te02 & \te02 \\ \hline
0 & 6 & 44 & \te02 & 12 & \te02 & \te02 & \te02 \\ \hline
0 & 7 & 22 & \te02 & 6 & 12 & \te02 & \te02 \\ \hline
0 & 8 & 70 & \te04 & \te02 & 6 & 6 & 6 \\ \hline
0 & 9 & 180 & \te04 & 10 & 6 & \te02 & 6 \\ \hline
1 & 1 & 44 & 6 & 6 & 92 & 10 & 6 \\ \hline
1 & 2 & 44 & 6 & 6 & 92 & 10 & 6 \\ \hline
1 & 5 & 24 & 4 & 16 & 24 & 32 & 4 \\ \hline
1 & 6 & 62 & \te02 & 16 & 8 & 6 & \te02 \\ \hline
1 & 7 & 80 & \te02 & 48 & 6 & 12 & 4 \\ \hline
1 & 8 & 114 & 4 & \te02 & 6 & 4 & \te04 \\ \hline
1 & 9 & 4 & \te02 & 16 & 18 & 4 & \te02 \\ \hline
2 & 2 & 18 & \ok & \ok & \ok & \ok & \ok \\ \hline
2 & 3 & \ok & \ok & \ok & \ok & \ok & \ok \\ \hline
2 & 4 & 18 & \ok & \ok & \ok & \ok & \ok \\ \hline
2 & 5 & 24 & 4 & 16 & 24 & 32 & 4 \\ \hline
2 & 9 & \ok & \ok & \ok & \ok & \ok & \ok \\ \hline
3 & 3 & \ok & \ok & \ok & \ok & \ok & \ok \\ \hline
3 & 4 & \ok & \ok & \ok & \ok & \ok & \ok \\ \hline
3 & 5 & \ok & \ok & \ok & \ok & \ok & \ok \\ \hline
3 & 6 & \ok & \ok & \ok & \ok & \ok & \ok \\ \hline
3 & 7 & \ok & \ok & \ok & \ok & \ok & \ok \\ \hline
3 & 8 & 62 & \te{-1}2 & 16 & 8 & 6 & \te{-1}2 \\ \hline
4 & 2 & \ok & \te32 & \ok & \ok & \ok & \ok \\ \hline
4 & 3 & \ok & \ok & \ok & \ok & \ok & \ok \\ \hline
4 & 4 & 28 & \ok & \ok & \ok & \ok & \ok \\ \hline
4 & 5 & \ok & \ok & \ok & \ok & \ok & \ok \\ \hline
4 & 6 & \ok & \ok & \ok & \ok & \ok & \ok \\ \hline
4 & 7 & \ok & \ok & \ok & \ok & \ok & \ok \\ \hline
4 & 8 & 28 & \ok & \ok & \ok & \ok & \ok \\ \hline
4 & 9 & \ok & \ok & \ok & \ok & \ok & \ok \\ \hline
5 & 2 & \ok & \ok & \ok & \ok & \ok & \ok \\ \hline
5 & 3 & \ok & \ok & \ok & \ok & \ok & \ok \\ \hline
5 & 4 & 24 & \ok & \ok & \ok & \ok & \ok \\ \hline
5 & 5 & \ok & \ok & \ok & \ok & \ok & \ok \\ \hline
5 & 6 & \ok & 4 & \ok & \ok & \ok & \ok \\ \hline
5 & 7 & 14 & \ok & \ok & \ok & \ok & \ok \\ \hline
5 & 8 & 14 & \ok & \ok & \ok & \ok & \ok \\ \hline
5 & 9 & \ok & 4 & \ok & \ok & \ok & \ok \\ \hline
6 & 2 & \ok & \ok & \ok & \ok & \ok & \ok \\ \hline
6 & 3 & \ok & \ok & \ok & \ok & \ok & \ok \\ \hline
6 & 4 & \ok & \ok & \ok & \ok & \ok & \ok \\ \hline
6 & 5 & \ok & \ok & \ok & \ok & \ok & \ok \\ \hline
6 & 6 & \ok & \ok & \ok & \ok & \ok & \ok \\ \hline
6 & 7 & 14 & \ok & \te{-3}2 & \ok & \ok & \ok \\ \hline
6 & 8 & 24 & \ok & \ok & \ok & \ok & \ok \\ \hline
6 & 9 & \ok & \ok & \ok & \ok & \ok & \ok \\ \hline
7 & 1 & 36 & \te12 & 8 & 4 & \te12 & \te12 \\ \hline
7 & 2 & 24 & \ok & \ok & \ok & \ok & \ok \\ \hline
7 & 3 & \ok & \ok & \ok & \ok & \ok & \ok \\ \hline
7 & 4 & \ok & \ok & \ok & \ok & \ok & \ok \\ \hline
7 & 5 & \ok & \ok & \ok & \ok & \ok & \ok \\ \hline
7 & 6 & \ok & \ok & \ok & \ok & \ok & \ok \\ \hline
7 & 7 & \ok & \ok & \ok & \ok & \ok & \ok \\ \hline
7 & 8 & 24 & \ok & \ok & \ok & \ok & \ok \\ \hline
7 & 9 & \ok & \ok & \ok & \ok & \ok & \ok \\ \hline
8 & 1 & 60 & 6 & 8 & 8 & 4 & \te12 \\ \hline
8 & 2 & 48 & \ok & \ok & \ok & \ok & \ok \\ \hline
8 & 3 & 44 & \ok & \ok & \ok & \ok & \ok \\ \hline
8 & 4 & 42 & \ok & \ok & \ok & \ok & \ok \\ \hline
8 & 5 & 20 & \ok & \ok & \ok & \ok & \ok \\ \hline
8 & 6 & 32 & \ok & \ok & \ok & \ok & \ok \\ \hline
8 & 7 & 40 & \ok & \ok & \ok & \ok & \ok \\ \hline
8 & 8 & 44 & \te22 & \ok & \ok & \ok & \ok \\ \hline
8 & 9 & 18 & \ok & \ok & \ok & \ok & \ok \\ \hline
9 & 1 & 18 & 4 & 4 & \te12 & \te12 & \te12 \\ \hline
9 & 2 & \ok & \ok & \ok & \ok & \ok & \ok \\ \hline
9 & 3 & \ok & \ok & \ok & \ok & \ok & \ok \\ \hline
9 & 4 & 28 & \ok & \ok & \ok & \ok & \ok \\ \hline
9 & 5 & \ok & \ok & \ok & \ok & \ok & \ok \\ \hline
9 & 6 & 24 & \ok & \ok & \ok & \ok & \ok \\ \hline
9 & 7 & \ok & \ok & \ok & \ok & \ok & \ok \\ \hline
9 & 8 & 32 & \ok & \ok & \ok & \ok & \ok \\ \hline
9 & 9 & \ok & \ok & \ok & \ok & \ok & \ok \\ \hline
\end{supertabular}
\end{center}


\bigskip
\begin{center}
\tablefirsthead{%
      \multicolumn{4}{c}{TABLE 7} \\
      \multicolumn{4}{c}{Values of $q$-power free constants $2\leq k\leq30$ for which the equation} \\ 
      \multicolumn{4}{c}{$G_n=kx^q$ has no solutions with $A=1$, $B=1$, and $q=3$, $5$} \\ 
      \multicolumn{4}{c}{}\\ \hline
      $G_0$ & $G_1$ & $q=3$ & \\ 
                    &              & $q=5$ & \\\hline\hline}
\tablehead{%
\hline
\multicolumn{4}{|l|}{Table 7: Impossible values of $k$ in sequences with $A=1$, $B=1$} \\
\multicolumn{4}{|l|}{\small\sl (continued from previous page)} \\
      \hline
      $G_0$ & $G_1$ & q=3 & \\ 
                 &            & q=5 & \\ \hline\hline}
\tabletail{%
      \hline
      \multicolumn{3}{|r|}{\small\sl continued on next page} \\
      \hline}
\tablelasttail{%
      \hline}
\begin{supertabular}{| c | c | c  l |}
1 & 3 &  & 5, 6, 9, 10, 12--15, 17, 19, 20--23, 26, 30 \\ 
  &  &  & 5, 6, 8--10, 12--15, 16, 17, 19--23, 25--28, 30 \\  \hline
1 & 4 & &  6, 10, 13, 15, 17, 18, 20, 22, 25, 26, 28, 29 \\ 
  & &  &  6, 8, 10, 11, 13, 15--18, 20--22, 25--30 \\ \hline
1 & 5 & & 9, 12--15, 18--20, 22, 23, 26, 29, 30 \\ 
  & &  &  2, 8, 9, 12--16, 18--23, 25, 26, 29, 30  \\ \hline
1 & 6 & & 2, 3, 10--12, 15, 17, 18, 23, 25, 26, 30 \\ 
  & &  &  2, 3, 8, 10, 12, 14--19, 21, 23, 25--30 \\ \hline
1 & 7 &  & 3, 4, 9, 10, 12--14, 17, 18, 21, 22, 25, 26, 28--30 \\
  & &  & 2--4, 9, 10, 12--14, 17--22, 25, 26, 28--30 \\ \hline
1 & 8 & & 2, 3, 5, 10--12, 15, 18, 21, 25, 28--30 \\ 
  & &  &  2--5, 10--12, 14--16, 18, 20--23, 25, 27, 28--30 \\ \hline
1 & 9 & &  4, 5, 11, 13, 14, 17, 18, 20, 21, 23, 25, 26 \\
  & & & 2, 4--6, 11--14, 16--18, 20, 21, 23, 25--28, 30 \\ \hline
2 & 5  & & 6, 10, 13, 15, 17, 18, 20, 22, 25, 26, 28, 29 \\
  & & & 6, 8, 10, 11, 13, 15--18, 20--22, 25--30 \\ \hline
2 & 6 &  & 3, 5, 9, 10--13, 15, 17, 18, 20, 21, 23, 25, 26, 28--30 \\ 
  & & & 3, 5, 7, 9--13, 15--21, 23, 25--27, 28--30 \\ \hline
2 & 7 & & 4, 6, 10, 12, 13--15, 18, 20, 22, 23, 26, 28, 29 \\ 
  & & & 6, 10, 12--15, 17, 18, 20--23, 26--29 \\ \hline
2 & 8 & & 5, 7, 9, 11, 12, 17, 19, 20, 23, 25, 26, 29, 30  \\ 
  & & & 3, 5, 7, 9, 11--13, 15--17, 19--23, 25--27, 29, 30 \\ \hline
2 & 9 & & 4, 6, 10, 14, 15, 18, 21--23, 25, 26, 28, 30 \\ 
  & & & 3, 4, 6, 8, 10, 13--16, 18, 19, 21--23, 25--28, 30  \\ \hline
3 & 7 & & 9, 12--15, 18--20, 22, 23, 26, 29, 30  \\ 
  & & & 2, 8, 9, 12--16, 18, 19--23, 25, 26, 29, 30  \\ \hline
3 & 8 & & 4, 6, 10, 12--15, 18, 20, 22, 23, 26, 28, 29 \\
  & & & 6, 10, 12--15, 17, 18, 20--23, 26--29 \\  \hline
3 & 9 & & 4, 5, 7, 10, 11, 13--15, 17--20, 23, 25, 26, 29, 30   \\ 
  & & & 2, 4, 5, 7, 8, 10, 11, 13--20, 22, 23, 25--30 \\ \hline
4 & 9 & & 2, 3, 10--12, 15, 17, 18, 23, 25, 26, 30 \\ 
  & & & 2, 3, 8, 10, 12, 14--19, 21, 23, 25--30 \\ \hline
6 & 4 & & 5, 7, 9, 11, 12, 17, 19, 20, 23, 25, 26, 29, 30 \\ 
  & & & 3, 5, 7, 9, 11--13, 15--17, 19--23, 25--27, 29, 30 \\ \hline
6 & 5 & & 3, 4, 9, 10, 12--14, 17, 18, 21, 22, 25, 26, 28--30 \\ 
  & & & 2, 4, 3, 9, 10, 12--14, 17--22, 25, 26, 28--30 \\ \hline
7 & 3 & & 2, 6, 9, 12, 14, 17, 18, 20--22, 25, 28--30 \\
  & & & 2, 5, 6, 8, 9, 12, 14, 16--19, 20--22, 25, 27--30 \\ \hline
7 & 4 & & 2, 6, 9, 12, 14, 17, 18, 20--22, 25, 28--30  \\ 
  & & & 2, 5, 6, 8, 9, 12, 14, 16--22, 25, 27--30 \\ \hline
7 & 5 & & 4, 6, 10, 14, 15, 18, 21--23, 25, 26, 28, 30  \\ 
  & & & 3, 4, 6, 8, 10, 13--16, 18, 19, 21--23, 25--28, 30  \\ \hline
8 & 2 & & 3, 5, 13, 15, 17--19, 21, 23, 25, 26, 28--30 \\ 
  & & & 3--5, 7, 9, 11, 13, 15--19, 21, 23, 25--30 \\ \hline
8 & 3 & & 2, 4, 6, 7, 9, 12, 15, 17, 19--23, 26, 28, 30 \\ 
  & & & 4, 6, 7, 9, 10, 12, 15--17, 19, 20--22, 23, 26--30 \\ \hline
8 & 4 & & 3, 5--7, 10, 11, 13, 15, 17--23, 25, 26, 29, 30 \\
  & & & 2, 3, 5--7, 9--11, 13--15, 17--23, 25--27, 29, 30 \\  \hline
8 & 5 & & 2, 4, 6, 7, 9, 12, 15, 17, 19--23, 26, 28, 30  \\ 
 & & & 4, 6, 7, 9, 10, 12, 15--17, 19--23, 26--30 \\ \hline
8 & 6 & & 3, 5, 13, 15, 17--19, 21, 23, 25, 26, 28--30 \\ 
  & & & 3--5, 7, 9, 11, 13, 15--19, 21, 23, 25--30 \\ \hline
8 & 7 & & 4, 5, 11, 13, 14, 17, 18, 20, 21, 23, 25, 26 \\ 
  & & & 2, 4--6, 11--14, 16--18, 20, 21, 23, 25--28, 30 \\ \hline
9 & 1 & & 2, 5--7, 12, 13, 15, 18--20, 23, 26, 28--30 \\ 
  & & & 2--7, 12--16, 18--20, 22, 23, 26--30\\ \hline
9 & 2 & & 4--6, 10, 12, 14, 15, 17, 20, 25, 26, 28--30 \\ 
  & & & 3--6, 8, 10, 12, 14, 15, 17--22, 25--30  \\ \hline
9 & 3 & & 2, 4, 5, 7, 10, 11, 13, 14, 17--20, 22, 23, 25, 26, 28--30 \\ 
  & & & 2, 4, 5, 7, 8, 10, 11, 13, 14, 16--20, 22, 23, 25, 26, 28--30 \\ \hline
9 & 4 & & 3, 6, 7, 10--12, 15, 18, 20, 22, 25, 26, 29 \\ 
  & & & 2, 3, 6, 7, 8, 10--12, 15, 16, 18, 20--23, 25--29 \\ \hline
9 & 5 & & 3, 6, 7, 10--12, 15, 18, 20, 22, 25, 26, 29 \\ 
  & & & 2, 3, 6--8, 10--12, 15, 16, 18, 20--23, 25--29 \\ \hline
9 & 6 & & 2, 4, 5, 7, 10, 11, 13, 14, 17--20, 22, 23, 25, 26, 28--30\\ 
  & & & 2, 4, 5, 7, 8, 10, 11, 13, 14, 16--20, 22, 23, 25, 26, 28--30 \\ \hline
9 & 7 & & 4--6, 10, 12, 14, 15, 17, 20, 25, 26, 28--30 \\
  & & & 3--6, 8, 10, 12, 14, 15, 17--22, 25--30 \\ \hline
9 & 8 & & 2, 5--7, 12, 13, 15, 18--20, 23, 26, 28--30  \\
  & & & 2--7, 12--16, 18--22, 23, 26--30 \\ \hline
\end{supertabular}
\end{center}

\bigskip
\begin{center}
\tablefirsthead{%
      \multicolumn{4}{c}{TABLE 8} \\
      \multicolumn{4}{c}{Values of $q$-power free constants $2\leq k\leq30$ for which the equation} \\ 
      \multicolumn{4}{c}{$G_n=kx^q$ has no solutions with $A=1$, $B=2$, and $q=3$, $5$} \\ 
       \multicolumn{4}{c}{}\\ 
       \hline
      $G_0$ & $G_1$ & $q=3$ & \\ 
                    &              & $q=5$ & \\ \hline\hline}
\tablehead{%
\hline
\multicolumn{4}{|l|}{Table 8: Impossible values of $k$ in sequences with $A=1$, $B=2$} \\
\multicolumn{4}{|l|}{\small\sl (continued from previous page)}\\
      \hline
       $G_0$ & $G_1$ & $q=3$ & \\ 
                    &              & $q=5$ & \\ \hline\hline}
\tabletail{%
      \hline
      \multicolumn{4}{|r|}{\small\sl continued on next page} \\
      \hline}
\tablelasttail{%
      \hline}
\begin{supertabular}{| c | c | c l |}
1 & 4 & & 3, 5, 9--11, 13, 15, 17--23, 25, 28--30 \\
 & & & 2, 3, 5, 7, 9--13, 15--17, 19--22, 24, 25, 27, 29, 30 \\ \hline
1 & 5 & & 3, 6, 9, 11--15, 18--23, 25, 26, 28--30 \\
 & & & 3, 4, 6, 8--15, 18--27, 29, 30 \\ \hline
1 & 6 & & 2--5, 7, 9--12, 14, 15, 17--19, 21--23, 25, 26, 28--30 \\ 
 & & & 2--5, 7, 9, 10--19, 21--23, 25--30 \\ \hline
1 & 7 & & 5, 6, 10, 11, 15, 17--22, 25, 26, 28--30  \\ 
 & & & 4--6, 8, 10--22, 24--30  \\ \hline
1 & 8 & & 2--7, 9, 11--15, 17, 18, 20--23, 25, 29, 30 \\ 
 & & & 2--7, 9, 11--25, 27--30 \\ \hline
1 & 9 & & 2, 3, 5--7, 10, 13--16, 18--21, 23--26, 28, 30  \\ 
  & & & 2, 3, 5--8, 10, 12--28, 30  \\ \hline
2 & 3 & & 5, 9--12, 14, 15, 17, 19--22, 25, 26, 29, 30 \\ 
 & & & 5, 6, 8, 10--12, 15, 17--23, 25, 26, 28--30  \\ \hline
2 & 5 & & 3, 4, 6, 7, 10, 11, 13--15, 17, 18, 20--23, 25, 28--30 \\ 
 & & & 3, 4, 7, 10--12, 14--17, 18, 21--30 \\ \hline
2 & 7 & & 3--6, 9, 10, 12--15, 17--19, 21--23, 26, 28--30 \\ 
 & & & 3--6, 9, 10, 12--19, 20--24, 27--30 \\ \hline
2 & 8 & & 4--7, 10, 11, 13, 15, 17--23, 25, 26, 29, 30 \\ 
 & & & 5--7, 9--11, 13--15, 17--27, 29, 30 \\ \hline
2 & 9 & & 3--5, 7, 10--12, 14, 15, 18--23, 25, 26, 29, 30  \\ 
 & & & 3--8, 10--12, 14--23, 25--30  \\ \hline
3 & 4 & & 2, 5--7, 9, 11--15, 17, 19--23, 25, 26, 28--30  \\
 & & & 2, 5--9, 11, 13--15, 17, 20--24, 25, 27--30  \\ \hline
3 & 8 & & 4--7, 10--13, 15, 17--19, 21--23, 25, 26, 29 \\
 & & & 2, 4--7, 9--13, 15--24, 26--29 \\ \hline
4 & 2 & & 3, 5--7, 9, 11--13, 15, 18, 19, 21--23, 25, 26, 28--30 \\
 & & & 3, 5--9, 11--13, 15--30  \\ \hline
4 & 3 & & 2, 5--7, 9, 10, 12--15, 19--23, 25, 26, 28--30 \\
 & & & 2, 5--10, 12--15, 18--30 \\ \hline
4 & 5 & & 2, 3, 6, 7, 9--12, 15, 17--22, 25, 26, 28--30 \\
 & & & 2, 3, 6--12, 14, 15, 17--19, 21, 22, 24--28, 30  \\ \hline
4 & 7 & & 2, 3, 5, 6, 9, 11, 13, 17--23, 26, 28, 30 \\ 
 & & & 2, 3, 5, 6, 8--14, 16, 17, 19--28, 30 \\ \hline
4 & 9 & & 2, 3, 5, 10--15, 18, 19, 21--23, 25, 26, 28, 29 \\
 & & & 3, 5--8, 10--14, 16, 18--23, 25--27, 29, 30 \\ \hline
5 & 1 & & 3, 4, 6, 7, 9, 10, 12, 14, 15, 17-- 21, 23, 26, 29, 30  \\ 
 & & & 3, 4, 6--10, 12, 14--30 \\ \hline
5 & 2 & & 3, 4, 6, 7, 10, 11, 13--15, 17, 18, 20--23, 25, 28--30 \\
 & & & 3, 4, 6--11, 13--15, 17--30 \\ \hline
5 & 3 & & 6, 7, 9--12, 14, 15, 17, 21--23, 25, 26, 28--30  \\
  & & & 4, 6--12, 14--18, 20--30 \\ \hline
5 & 4 & & 2, 3, 6, 7, 9--12, 15, 17--21, 23, 25, 26, 28--30  \\
 & & & 2, 3, 6--13, 15, 17--21, 23--30  \\ \hline
5 & 6 & & 3, 9--15, 17, 19--23, 25, 26, 29, 30 \\
 & & & 2--4, 7--15, 17--27, 29, 30 \\ \hline
5 & 8 & & 2, 3, 6, 7, 9--11, 13, 15, 17, 19--23, 25, 26, 28--30 \\ 
 & & & 2, 3, 6, 7, 9--16, 19--29\\ \hline
6 & 1 & & 2--5, 7, 9--12, 14, 17--19, 21--23, 25, 26, 28--30 \\ 
 & & & 2--5, 7--12, 14, 16--30 \\ \hline
6 & 2 & & 5, 7, 9--13, 15, 17, 19-23, 25, 29, 30 \\ 
 & & & 5, 7--13, 15--17, 19--30 \\ \hline
6 & 3 & & 2, 4, 5, 7, 9--11, 13, 14, 17--20, 22, 23, 25, 26, 28, 29 \\ 
 & & & 2, 4, 5, 7--14, 16--20, 22--30 \\ \hline
6 & 4 & & 5, 9--12, 14, 15, 17, 19--22, 25, 26, 29, 30 \\
 & & & 2, 3, 5, 7--15, 17--23, 25--30 \\ \hline
6 & 5 & & 2, 3, 7, 9--14, 18--23, 25, 28--30 \\
 & & & 2--4, 7--15, 18--26, 28--30   \\ \hline
6 & 7 & & 2, 3, 5, 10--15, 17, 18, 20, 21, 23, 25, 26, 28--30 \\ 
 & & & 2--5, 8--15, 17, 18, 20--30  \\ \hline
6 & 9 & & 4, 5, 7, 10, 11, 13--15, 17, 19, 20, 22, 23, 25, 26, 28--30 \\ 
 & & & 2--5, 7, 8, 10--20, 22--26, 28--30 \\ \hline
7 & 1 & & 2, 6, 9--14, 18--23, 25, 26, 28--30 \\
 & & & 2, 6, 8--14, 16, 18--30 \\ \hline
7 & 2 & & 3--6, 9--15, 17--19, 21--23, 25, 26, 28, 30  \\
 & &  & 3--6, 8--15, 17--19, 21--30 \\ \hline
7 & 3 & & 4--6, 9--15, 18--22, 25, 28--30 \\
 & & & 4--6, 8--16, 18--22, 24--30  \\ \hline
7 & 4 & & 2, 3, 5, 6, 9--11, 13--15, 17, 19--22, 25, 28--30 \\
 & & & 2, 3, 5, 6, 8--17, 19--25, 27--30 \\ \hline
7 & 5 & & 2, 3, 6, 9--15, 17, 18, 21, 22, 25, 28, 30 \\
 & & & 2, 3, 6, 8--18, 20--28, 30 \\ \hline
7 & 6 & & 2, 3, 5, 10--15, 18, 19, 21--23, 25, 26, 28, 29 \\
 & & & 2--5, 8--15, 17--19, 21--30  \\ \hline
7 & 8 & & 2, 3, 5, 6, 9--15, 17--21, 23, 25, 28--30 \\
 & & & 2--6, 9--15, 17--21, 23--30  \\ \hline
8 & 1 & & 2--7, 9--15, 18, 20--23, 25, 26, 29, 30 \\
 & & & 2--7, 9--16, 18, 20--30  \\ \hline
8 & 2 & & 4--7, 9, 10, 12--15, 17, 19--21, 23, 25, 26, 28--30 \\
 & & & 4--7, 9--17, 19--21, 23--30  \\ \hline
8 & 3 & & 2, 5--7, 9, 11--15, 17, 18, 21--23, 26, 28-30 \\
 & & & 2, 4-7, 9--18, 20--24, 26--30  \\ \hline
8 & 4 & &  3, 6, 7, 9--15, 17--19, 21--23, 25, 26, 29, 30 \\
 & & & 3, 6, 7, 9--19, 21--27, 29, 30  \\ \hline
8 & 5 & & 2, 3, 6, 7, 9--11, 13--15, 17--20, 22, 23, 26, 28--30  \\ 
 & & & 2--4, 6, 7, 9--20, 22--30  \\ \hline
8 & 6 & & 2--5, 7, 9--15, 17--21, 23, 25, 26, 28, 30 \\
 & & & 2--5, 7, 9--21, 23--30  \\ \hline
8 & 7 & & 2, 3, 5, 6, 9--15, 17, 18, 20--22, 25, 26, 28--30 \\
 & & & 2--6, 9--15, 17--22, 24--30  \\ \hline
9 & 1 & & 2, 3, 5--7, 10--15, 17, 18, 20, 22, 23, 25, 26, 28--30 \\
 & & & 2, 3, 5, 6, 8, 10--18, 20, 22--30 \\ \hline
9 & 2 & & 4--7, 10--13, 15, 17--19, 21--23, 25, 26, 29 \\  
 & & & 3--8, 10--19, 21--23, 25--30 \\ \hline
9 & 3 & & 2, 4, 5, 7, 10, 11, 13--15, 17--19, 20, 22, 23, 25, 26, 28, 29, 30 \\
 & & & 2, 4, 5, 7, 8, 10--20, 22, 23, 24--26, 28--30  \\ \hline
9 & 4 & & 2, 3, 5--7, 10--15, 17--19, 21, 23, 25, 26, 28, 29 \\
 & & & 2, 3, 5--8, 10--21, 23--29  \\ \hline
9 & 5 & & 3, 4, 6, 7, 10--15, 17--22, 25, 26, 28, 29 \\
 & & & 3, 4, 6--8, 10--22, 24--30  \\ \hline
9 & 6 & & 2, 4, 5, 7, 10, 11, 13--15, 17--23, 25, 26, 28--30 \\
 & & & 2--5, 7, 8, 10--23, 25--30  \\ \hline
9 & 7 & & 2, 6, 10--15, 17--23, 26 \\
 & & & 2, 6, 8, 10--24, 26--30  \\ \hline
9 & 8 & & 2, 3, 5--7, 10--15, 17--20, 22, 23, 25, 28--30 \\
 & & & 2--7, 10--13, 15, 17--25, 27--30  \\ \hline
\end{supertabular}
\end{center}

\bigskip
\begin{center}
\tablefirsthead{%
      \multicolumn{4}{c}{TABLE 9} \\ 
      \multicolumn{4}{c}{Values of $q$-power free constants $2\leq k\leq30$ for which the equation} \\ 
      \multicolumn{4}{c}{$G_n=kx^q$ has no solutions with $A=2$, $B=1$, and $q=3$, $5$} \\ 
      \multicolumn{4}{c}{} \\ 
      \hline
       $G_0$ & $G_1$ & $q=3$ &  \\ 
                   &               & $q=5$ &  \\ \hline\hline}
\tablehead{%
\hline
\multicolumn{4}{|l|}{Table 9: Impossible values of $k$ in sequences with $A=2$, $B=1$} \\ 
\multicolumn{4}{|l|}{\small\sl (continued from previous page)}\\
      \hline
      $G_0$ & $G_1$ & $q=3$ &  \\ 
                   &               & $q=5$ &  \\ \hline\hline}
\tabletail{%
      \hline
      \multicolumn{4}{|r|}{\small\sl continued on next page} \\
      \hline}
\tablelasttail{%
      \hline}
\begin{supertabular}{| c | c | c l |}
1 & 1 & & 2, 5, 6, 9--15, 18--23, 25, 26, 28--30\\
& & & 2, 4--6, 8--16, 18--30 \\ \hline
1 & 3 & & 2, 5, 6, 9--15, 18--23, 25, 26, 28-30 \\
& & & 2, 4--6, 8--16, 18--30 \\ \hline
1 & 4 & & 5--7, 10--15, 17, 18, 20, 21, 23, 25, 26, 28--30 \\
 & & & 5--7, 10--18, 20, 21, 23--30 \\ \hline
1 & 5 & & 2, 6, 7, 9, 10, 12, 14, 15, 17--23, 25, 26, 28--30 \\
 & & & 2, 4, 6--10, 12, 14--26, 28--30 \\ \hline
1 & 6 & & 2, 3, 5, 9--12, 14, 15, 17, 19--23, 25, 26, 28--30 \\
 & & & 2, 3, 5, 8--12, 14--17, 19--25, 27--30 \\ \hline
1 & 7 & & 2, 3, 6, 10--14, 17--22, 25, 26, 28--30 \\
 & & & 2--4, 6, 8, 10--14, 16--22, 24--30 \\ \hline
1 & 8 & & 2, 3, 5, 7, 9, 10, 12-15, 18-- 22, 23, 25, 26, 29, 30 \\
 & & & 2--5, 7, 9, 10, 12--16, 18--27, 29, 30 \\ \hline
1 & 9 & & 2, 3, 5, 6, 10--12, 14, 15, 17, 18, 20--23, 25, 26, 28--30 \\
 & & & 2--6, 8, 10--12, 14--18, 20--30 \\ \hline
2 & 3 & & 5--7, 10--15, 17, 18, 20, 21, 23, 25, 26, 28--30 \\
 & & & 5--7, 10--18, 20, 21, 23--30 \\ \hline
2 & 7 & & 5, 6, 9, 10, 12--15, 17, 18, 20--23, 25, 28--30 \\
 & & & 5, 6, 8--10, 12--15, 17--25, 27--30 \\ \hline
2 & 8 & & 3, 5, 7, 9--15, 17, 19, 20--23, 25, 26, 28--30 \\
 & & & 3, 5, 7, 9--15, 17, 19--30 \\ \hline
2 & 9 & & 3, 6, 7, 10--15, 17--19, 22, 23, 25, 26, 28--30 \\
 & & & 3, 4, 6, 7, 10--19, 22--30 \\ \hline
3 & 3 & & 2, 5--7, 10, 12--15, 17--20, 22, 23, 25, 26, 28--30 \\
 & & & 2, 4--8, 10--20, 22--30 \\ \hline
3 & 4 & & 5, 6, 9, 10, 12--15, 17, 18, 20--23, 25, 28--30 \\
 & & & 5, 6, 8--10, 12--15, 17--25, 27--30 \\ \hline
3 & 5 & & 2, 6, 7, 9, 10, 12, 14, 15, 17--23, 25, 26, 28--30 \\
 & & & 2, 4, 6--10, 12, 14--26, 28--30 \\ \hline
3 & 9 & & 2, 5--7, 10, 12--15, 17--20, 22, 23, 25, 26, 28--30 \\
& & & 2, 4--8, 10--20, 22--30  \\ \hline
4 & 3 & & 2, 6, 11--13, 15, 17--22, 25, 26, 28--30 \\
 & & & 2, 6--9, 11--13, 15--22, 24--30 \\ \hline
4 & 4 & & 2, 3, 5--7, 9--11, 13--15, 17--20, 21--23, 25, 26, 29, 30 \\
 & & & 2, 3, 5--11, 13--27, 29, 30 \\ \hline
4 & 5 & & 2, 6, 11--13, 15, 17--22, 25, 26, 28--30 \\
 & & & 2, 6--9, 11--13, 15--22, 24--30 \\ \hline
4 & 6 & & 3, 5, 7, 9--15, 17, 19--23, 25, 26, 28--30 \\
 & & & 3, 5, 7, 9--15, 17, 19--30 \\ \hline
4 & 7 & & 2, 3, 5, 9--12, 14, 15, 17, 19--23, 25, 26, 28--30 \\
 & & & 2, 3, 5, 8--12, 14--17, 19--25, 27--30 \\ \hline
5 & 3 & & 2, 6, 9, 10, 12--15, 17, 18, 20--23, 26, 28--30 \\
 & & & 2, 4, 6, 8--10, 12--18, 20--24, 26--30 \\ \hline
5 & 4 & & 2, 3, 7, 9--12, 14, 15, 18, 19--21, 23, 25, 26, 28, 29 \\
 & & & 2, 3, 7--12, 14--16, 18--29 \\ \hline
5 & 5 & & 2, 3, 6, 7, 9--14, 17--23, 25, 26, 28--30 \\
 & & & 2--4, 6--14, 16--30 \\ \hline
5 & 6 & & 2, 3, 7, 9--12, 14, 15, 18--21, 23, 25, 26, 28, 29 \\
 & & & 2, 3, 7--12, 14--16, 18--29 \\ \hline
5 & 7 & & 2, 6, 9, 10, 12--15, 17, 18, 20--23, 26, 28--30 \\
 & & & 2, 4, 6, 8--10, 12--18, 20--24, 26--30 \\ \hline
5 & 8 & & 3, 6, 7, 10--15, 17--19, 22, 23, 25, 26, 28--30 \\ 
 & & & 3, 4, 6, 7, 10--19, 22--30 \\ \hline
5 & 9 & & 2, 3, 6, 10--14, 17--22, 25, 26, 28--30 \\
 & & & 2, 3, 4, 6, 8, 10--14, 16--19, 20--22, 24-- 30 \\ \hline
6 & 2 & & 3, 5, 7, 9, 11--15, 17--21, 23, 25, 28--30 \\
 & & & 3--5, 7--9, 11--21, 23--25, 27--30 \\ \hline
6 & 3 & & 2, 5, 7, 10, 11, 13--15, 17--23, 25, 26, 28--30 \\
 & & & 2, 4, 5, 7, 8, 10, 11, 13--23, 25, 26, 28--30 \\ \hline
6 & 4 & & 2, 3, 5, 7, 9--13, 15, 17--21, 23, 25, 26, 28--30 \\
 & & & 2, 3, 5, 7, 9--13, 15--21, 23--30 \\ \hline
6 & 5 & & 3, 9--15, 17--19, 21--23, 25, 26, 28--30 \\
 & & & 2--4, 8--15, 17--19, 21--30 \\ \hline
6 & 6 & & 2, 3, 5, 7, 9--15, 17, 19--21, 23, 25, 26, 28--30 \\
 & & & 2--5, 7--17, 19--30 \\ \hline
6 & 7 & & 3, 9--15, 17--19, 21--23, 25, 26, 28--30 \\
 & & & 2--4, 8--15, 17--19, 21--30 \\ \hline
6 & 8 & & 2, 3, 5, 7, 9--13, 15, 17--21, 23, 25, 26, 28--30 \\
 & & & 2, 3, 5, 7, 9--13, 15--21, 23--30 \\ \hline
6 & 9 & & 2, 5, 7, 10, 11, 13--15, 17--23, 25, 26, 28--30 \\
 & & & 2, 4, 5, 7, 8, 10, 11, 13--23, 25, 26, 28--30 \\ \hline
7 & 2 & & 5, 6, 9, 10, 13--15, 17--23, 25, 26, 28--30\\
 & & & 3--6, 8--10, 13--23, 25--30 \\ \hline
7 & 3 & & 2, 5, 6, 9, 10, 12, 14, 15, 17--23, 25, 26, 28, 30 \\
 & & & 2, 4, 5, 6, 8--10, 12, 14--28, 30 \\ \hline
7 & 4 & & 2, 3, 5, 6, 9, 11--14, 17--23, 26, 28--30 \\
 & & & 3, 5, 6, 8, 9, 11--14, 16--26, 28--30 \\ \hline
7 & 5 & & 2, 3, 6, 10--15, 18--23, 26, 28--30 \\
 & & & 2--4, 6, 8, 10--16, 18--24, 26--30 \\ \hline
7 & 6 & & 3, 5, 9--14, 15, 17, 18, 20--22, 25, 26, 28--30 \\
 & & & 2--5, 9--18, 20--22, 24--30 \\ \hline
7 & 7 & & 2, 3, 5, 6, 9--15, 17--20, 22, 23, 25, 26, 28--30 \\
 & & & 2--6, 8--20, 22--30 \\ \hline
7 & 8 & & 3, 5, 9--15, 17, 18, 20--22, 25, 26, 28--30 \\
 & & & 2--5, 9--18, 20--22, 24--30  \\ \hline
7 & 9 & & 2, 3, 6, 10--15, 18--23, 26, 28--30 \\
 & & & 2--4, 6, 8, 10--16, 18--24, 26--30 \\ \hline
8 & 1 & & 2, 3, 5--7, 9, 11--14, 17--20, 22, 23, 25, 26, 28--30 \\
 & & & 2, 4, 5, 6, 7, 9, 11--14, 16--20, 22--30 \\ \hline
8 & 2 & & 3, 5--7, 9--11, 13, 15, 17--23, 25, 28--30\\
 & & & 3--7, 9--11, 13, 15--25, 27--30 \\ \hline
8 & 3 & & 2, 5--7, 9--12, 15, 17--23, 25, 26, 28--30 \\ 
 & & & 2, 4--7, 9--12, 15--30 \\ \hline
8 & 4 & & 3, 5--7, 9, 10, 13--15, 17--22, 25, 26, 28--30 \\
 & & & 2, 3, 5--7, 9--11, 13--15, 17--30 \\ \hline
8 & 5 & & 2, 3, 6, 7, 9, 10, 12--15, 17, 19--23, 25, 26, 28, 29 \\
 & & & 2--4, 6, 7, 9, 10, 12--17, 19--29 \\ \hline
8 & 6 & & 2, 3, 5, 7, 9, 11--13, 15, 17, 19, 21--23, 25, 26, 29, 30\\
& & & 2, 3, 4, 7, 9, 11--19, 21--27, 29, 30 \\ \hline
8 & 7 & & 2, 3, 5, 6, 10--15, 17--21, 23, 25, 28--30 \\
 & & & 2--6, 10--21, 23--25, 27--30 \\ \hline
8 & 8 & & 2, 5, 6, 9--15, 18--23, 25, 26, 28--30 \\
 & & & 2--7, 9--23, 25--30 \\ \hline
8 & 9 & & 2, 3, 5, 6, 10--15, 17--21, 23, 25, 28--30 \\
 & & & 2--6, 10--21, 23--25, 27--30  \\ \hline
9 & 1 & & 2, 3, 5--7, 10, 12--15, 18--22, 25, 26, 28--30 \\
 & & & 2--8, 10, 12--16, 18--22, 24--30 \\ \hline
9 & 2 & & 3, 5--7, 10--12, 14, 15, 17--23, 25, 26, 29, 30 \\
 & & & 3--8, 10--12, 14, 15, 17--27, 29, 30 \\ \hline
9 & 3 & & 2, 5--7, 10--14, 17--23, 25, 26, 28--30 \\
 & & & 2, 4--8, 10--14, 16--30 \\ \hline
9 & 4 & & 2, 3, 5--7, 10, 12, 13, 15, 18--23, 25, 26, 29, 30 \\
 & & & 2, 3, 5, 6, 8, 10--13, 15, 16, 18--30  \\ \hline
9 & 5 & & 2, 3, 6, 7, 10--12, 14, 15, 17, 18, 20--23, 25, 26, 28--30\\
 & & & 2--4, 6--8, 10--12, 14--18, 20--30 \\ \hline
9 & 6 & & 2, 3, 5, 10, 11, 13--15, 17--20, 22, 23, 25, 26, 28--30 \\
 & & & 2--5, 7, 8, 10, 11, 13--20, 22--30 \\ \hline
9 & 7 & & 2, 3, 5, 6, 10, 12--15, 17--22, 25, 26, 28--30 \\
 & & & 2--6, 8, 10, 12--22, 24--30 \\ \hline
9 & 8 & & 2, 3, 5--7, 11--15, 17--23, 26, 28, 30 \\
 & & & 2--7, 11--24, 26--28, 30 \\ \hline
9 & 9 & & 2, 3, 5--7, 10--15, 17--23, 25, 26, 28--30\\
 & &  & 2--8, 10--26, 28--30 \\ \hline
\end{supertabular}
\end{center}

\bigskip
\begin{center}
\tablefirsthead{%
      \multicolumn{4}{c}{TABLE 10} \\ 
      \multicolumn{4}{c}{Values of $q$-power free constants $2\leq k\leq30$ for which the equation} \\ 
      \multicolumn{4}{c}{$G_n=kx^q$ has no solutions with $A=1$, $B=3$, and $q=3$, $5$} \\ 
      \multicolumn{4}{c}{}\\
      \hline
      $G_0$ & $G_1$ & $q=3$ & $q=5$\\ \hline\hline}
\tablehead{%
\hline
\multicolumn{4}{|l|}{Table 10: Impossible values of $k$ in sequences with $A=1$, $B=3$} \\ 
\multicolumn{4}{|l|}{\small\sl (continued from previous page)}\\
      \hline
      $G_0$ & $G_1$ & $q=3$ & $q=5$ \\ \hline\hline}
\tabletail{%
      \hline
      \multicolumn{4}{|r|}{\small\sl continued on next page} \\
      \hline}
\tablelasttail{%
      \hline}
\begin{supertabular}{| c | c | l | l |}
1 & 2 & 29 & 3, 4, 6--8, 10, 12--14, 16--25, 27--30  \\ \hline
1 & 3 & 22 & 2, 4, 5, 7--14, 16, 18--26, 28--30 \\ \hline
1 & 5 & 2, 6 & 2--4, 6, 7, 9--22, 24--26, 28--30 \\ \hline
1 & 6 & 29 & 2--5, 7, 8, 10--26, 28--30 \\ \hline
 1 & 7 & 30 & 3--6, 8, 9, 11--30 \\ \hline
1 & 8 & 29, 30 & 2--7, 9, 10, 12--27, 29, 30 \\ \hline
1 & 9 & 11, 20, 30 & 2--5, 7, 8, 10, 11, 13--30 \\ \hline
 2 & 3 & 25 & 4--8, 10--17, 19--21, 22, 24--30 \\ \hline
 2 & 4 & 25 & 3, 5--9, 11--17, 19--21, 23--29 \\ \hline
2 & 6 & 13, 29 & 3--5, 7--11, 13--29 \\ \hline
2 & 7 & 6, 10 & 3--6, 8--12, 14--26, 28--30 \\ \hline
2 & 9 & 5, 13 & 3--8, 10--14, 16--26, 28--30 \\ \hline
3 & 1 & 4, 9, 15, 22, 23 & 2, 4--9, 11, 12, 14--30 \\ \hline
3 & 2 & 5, 22, 25, 26 & 4--10, 12--16, 18--30 \\ \hline
3 & 4 &  10, 19, 26 & 5--12, 14--24, 26--30 \\ \hline
3 & 5 & 12, 19, 26 & 2, 4, 6--8, 10--13, 15--24, 26--28, 30 \\ \hline
3 & 7 & & 4--6, 8--12, 14, 15, 17--23, 25--30 \\ \hline
3 & 8 & 2, 5, 15, 22, 30& 2, 4--7, 9--16, 18--30 \\ \hline
4 & 1 &  6, 12, 18 & 2, 3, 5--12, 14, 15, 17--30 \\ \hline 
4 & 2 & 23 & 3, 5--13, 15--19, 21--30 \\ \hline
4 & 3 & 5, 6, 11, 18 & 2, 5--14, 16--23, 25--30 \\ \hline
4 & 5 & 6, 10, 12, 14, 23 & 2, 3, 6--16, 18--30 \\ \hline
4 & 6 & & 2, 3, 5, 7--17, 19--30 \\ \hline
4 & 8 & 25 & 2, 3, 5--7, 9, 10--19, 21--30 \\ \hline
4 & 9 & 13 & 2, 3, 5, 7, 8, 10, 12--20, 22--30 \\ \hline
5 & 1 & 3, 30 & 2--4, 6--15, 17, 18, 20--30 \\ \hline
5 & 2 & 3, 29 & 3, 4, 6--16, 18--22, 24--30 \\ \hline
5 & 3 & 22 & 2, 4, 6--17, 19--26, 28--30 \\ \hline
5 & 4 & 3, 10, 25 & 2, 3, 6--17, 18, 20--30 \\ \hline
5 & 6 & 4, 18, 30 & 2--4, 7--20, 22--30 \\ \hline
5 & 7 & 11, 26 & 2--4, 6, 8--21, 23--30 \\ \hline
5 & 9 & 6, 10, 11, 15, 26 & 2, 4, 6--8, 10--23, 25--28, 30 \\ \hline
6 & 1 & 3, 4, 13, 23, 25 & 2--5, 7--18, 20, 21, 23--30 \\ \hline
6 & 2 & 18, 19, 29, 30 & 3--5, 7--19, 21--25, 27--30 \\ \hline
6 & 3 & 5, 10, 23 & 2, 4, 5, 7--20, 22--29 \\ \hline
6 & 4 &  2, 10, 17, 19 & 2, 3, 5, 7--21, 23--30 \\ \hline
6 & 5 & 12, 25, 26 & 2--4, 7--22, 24--30 \\ \hline
6 & 7 &  10, 22 & 2--5, 8--24, 26--30 \\ \hline
6 & 8 & 17, 20, 23 & 2, 3, 5, 7, 9--25, 27--30 \\ \hline
7 & 1 & 9, 18, 26 & 4--6, 8, 9--21, 23, 24, 26--30 \\ \hline
7 & 2 & 13, 30 & 3--6, 8--22, 24--28, 30 \\ \hline
7 & 3 & 5, 19 & 2, 4--6, 8--23, 25--30 \\ \hline
7 & 4 & & 2, 3, 5, 6, 8--24, 26--30 \\ \hline
7 & 5 & 13 & 2--4, 6, 8--25, 27--30 \\ \hline
7 & 6 & 4, 15, 22, 26 & 2--5, 8--26, 28--30 \\ \hline
7 & 8 & 23 & 2, 3, 5, 6, 9--28, 30 \\ \hline
7 & 9 & & 2--6, 8, 10--29 \\ \hline
8 & 1 & 4, 5, 19, 26 & 2--7, 9--24, 26, 27, 29, 30 \\ \hline
8 & 2 & 3, 12, 13 & 3--7, 9--25, 27--30 \\ \hline
8 & 3 & 13, 17 & 2, 4--6, 7, 9--17, 19--26, 28--30 \\ \hline
8 & 4 &  15 & 2, 3, 5--7, 9--27, 29, 30 \\ \hline
8 & 5 & 4, 20, 30 & 2, 4, 7, 9--28, 30 \\ \hline
8 & 6 & 10, 12, 22, 25 & 2--5, 7, 10--29 \\ \hline
8 & 7 & 3, 5, 6, 10, 15, 18, 20, 21 & 2--6, 9--22, 24--30 \\ \hline
8 & 9 & 10, 22 & 2--7, 10--30 \\ \hline
9 & 1 & 23 & 2--8, 10--27, 29, 30 \\ \hline
9 & 2 & 17, 18, 19, 20, 25 & 3--8, 10--28, 30 \\ \hline
9 & 3 & 5, 11, 12, 22 & 4--8, 10--29 \\ \hline
9 & 4 & 18, 22, 29 & 2, 3, 5--8, 10--26, 28--30 \\ \hline
9 & 5 & 2, 15, 26, 29, 30 & 2--4, 6--8, 10--30 \\ \hline
9 & 6 &  3, 5, 13, 15, 23 & 2--5, 7, 8, 10--30 \\ \hline
9 & 7 & 6 & 2--6, 8, 10--30 \\ \hline
9 & 8 &  10, 11, 22, 23, 30 &  2--7, 10--30 \\ \hline
\hline
\end{supertabular}
\end{center}

\bigskip
\begin{center}
\tablefirsthead{%
      \multicolumn{4}{c}{TABLE 11} \\
      \multicolumn{4}{c}{Values of $q$-power free constants $2\leq k\leq30$ for which the equation} \\ 
      \multicolumn{4}{c}{$G_n=kx^q$ has no solutions with $A=2$, $B=2$, and $q=3$, $5$} \\ 
     \multicolumn{4}{c}{}\\
     \hline
      $G_0$ & $G_1$ & $q=3$ &  \\
                    &              & $q=5$ & \\ \hline\hline}
\tablehead{%
\hline
\multicolumn{4}{|l|}{Table 11: Impossible values of $k$ in sequences with $A=2$, $B=2$} \\ 
\multicolumn{4}{|l|}{\small\sl continued from previous page}\\
      \hline
      $G_0$ & $G_1$ & $q=3$ &  \\
                    &              & $q=5$ & \\ \hline\hline}
\tabletail{%
      \hline
      \multicolumn{4}{|r|}{\small\sl continued on next page} \\
      \hline}
\tablelasttail{%
      \hline}
\begin{supertabular}{| c | c | c  l |}
1 & 1 & & 2, 3, 5--7, 9, 11--13, 15, 17, 18, 20--23, 25, 29, 30  \\
 & & & 2, 3, 5--9, 11--15, 17--27, 29, 30 \\ \hline
1 & 5 & & 2--4, 6, 7, 9--11, 13, 15, 17, 19--23, 26, 28--30 \\
 & & & 2--4, 6--11, 13--30 \\ \hline
1 & 6 & & 3, 7, 9--11, 13, 15, 17--19, 21--23, 25, 28--30  \\
 & & & 3--5, 7--13, 15--30 \\ \hline
1 & 7 & & 3--6, 9--15, 17--19, 21--23, 25, 28--30 \\
 & & & 3--6, 9--15, 17--28, 30 \\ \hline
1 & 8 & & 2, 5, 7, 9--15, 17, 19, 21--23, 25, 26, 28--30  \\
 & & & 2, 5-7, 9--11, 13--17, 19--30 \\ \hline
1 & 9 & & 2, 4--7, 10--15, 17--19, 21--23, 25, 26, 29, 30  \\
 & & & 2, 4, 5--8, 10--19, 21--30 \\ \hline
2 & 2 & & 3--6, 9, 10--15, 17, 18, 21--23, 25, 26, 29, 30  \\
 & & & 3--7, 9--12, 14--19, 21--30 \\ \hline
2 & 3 & & 5--7, 11, 13, 15, 17--19, 21--23, 25, 28--30 \\
 & & & 4--9, 11--15, 17--25, 27--30 \\ \hline
2 & 5 & & 3, 6, 9--12, 15, 17--23, 25, 26, 28--30  \\
 & & & 3, 4, 6, 7, 9, 10--13, 15, 17--19, 21--27, 29, 30 \\ \hline
2 & 7 & & 3, 5, 6, 9, 11, 13--15, 19--22, 25, 26, 28--30 \\
 & & & 3--6, 8--15, 17, 19--30 \\ \hline
2 & 8 & & 3--6, 9--15, 17, 18, 21--23, 25, 26, 29, 30  \\
 & & & 3--7, 9--12, 14--19, 21--30 \\ \hline
2 & 9 & & 3--7, 10, 11, 13--15, 17--19, 23, 25, 26, 29, 30  \\
 & & & 3--8, 10--21, 23--30 \\ \hline
3 & 2 & & 4--7, 9, 11--15, 17--22, 25, 26, 29, 30 \\
 & & & 4--9, 11--23, 25--30 \\ \hline
3 & 3 & & 2, 4--7, 9--11, 13--15, 17--23, 25, 26, 28, 29  \\
 & & & 2, 4--11, 13--29 \\ \hline
3 & 4 & & 2, 5--7, 9--13, 15, 17--19, 21--23, 25, 26, 28--30  \\
 & & & 2, 5--13, 15--30 \\ \hline
3 & 5 & & 7, 9--15, 17, 19--23, 25, 28--30 \\
 & & & 4, 7--15, 17--26, 28--30 \\ \hline
3 & 7 & & 5, 9, 11, 12, 14, 15, 17, 19, 21--23, 25, 26, 28--30 \\
 & & & 2, 4--6, 8--15, 17--19, 21--23, 25--30 \\ \hline
4 & 1 & & 2, 3, 5--7, 9, 11--15, 17--21, 23, 25, 26, 29, 30 \\
 & & & 3, 5--9, 11--30 \\ \hline
4 & 2 & & 6, 7, 11, 13--15, 17--23, 25, 26, 29, 30  \\
 & & & 6--8, 10, 11, 13--27, 29, 30 \\ \hline
4 & 3 & & 2, 5--7, 10, 11, 13, 15, 17--19, 21--23, 25, 26, 28--30   \\
 & & & 2, 5--7, 9--13, 15--30 \\ \hline
4 & 4 & & 3, 6, 9--12, 15, 17--23, 25, 26, 28--30  \\
 & & & 3, 6, 8--12, 14, 15, 17--25, 27--30 \\ \hline
4 & 5 & & 3, 7, 9--11, 13--15, 17, 19--23, 25, 26, 29, 30  \\
 & & & 3, 7--17, 19--30 \\ \hline
4 & 6 & & 7, 9--15, 17, 19, 21--23, 25, 26, 29, 30  \\
 & & & 2, 7--19, 21--30  \\ \hline
4 & 7 & & 2, 3, 5, 6, 9--15, 17--19, 21, 23, 25, 26, 28--30  \\
 & & & 2, 3, 6, 8--15, 17--21, 23--30 \\ \hline
4 & 9 & & 5--7, 11, 13--15, 17--20, 22, 23, 25, 28, 30  \\
 & & & 3, 5, 7, 8, 10--15, 17--25, 27--30 \\ \hline
5 & 1 & & 2--4, 6, 9--11, 13--15, 17--23, 25, 28--30  \\
 & & & 2--4, 6, 8--11, 13--25, 27--30 \\ \hline
5 & 2 & & 3, 6, 7, 9--13, 15, 17--23, 25, 26, 28, 30 \\
 & & & 3, 6--13, 15, 17--30 \\ \hline
5 & 3 & & 7, 9--15, 17--19, 21--23, 25, 26, 29, 30 \\
 & & & 2, 4, 7--15, 17--19, 21--24, 26--30 \\ \hline
5 & 4 & & 2, 6, 9--15, 17, 19--23, 25, 26, 28--30  \\
 & & & 2, 6, 8--17, 19--30  \\ \hline
5 & 5 & & 2--4, 6, 7, 9--15, 17--19, 21--23, 25, 26, 28--30  \\
 & & & 2--4, 6--19, 21-30 \\ \hline
5 & 6 & & 3, 4, 9--15, 17--21, 23, 25, 26, 28--30  \\
 & & & 3, 4, 7, 8, 10--21, 23--26, 27--30 \\ \hline
5 & 7 & & 6, 9--11, 13--15, 17--19, 21--23, 25, 26, 28--30 \\
 & & & 2, 3, 6, 8--23, 25--30 \\ \hline
5 & 8 & & 2, 3, 6, 7, 9--15, 17--23, 25, 29, 30  \\
 & & & 2, 3, 6, 7, 9--15, 17--25, 27--30 \\ \hline
5 & 9 & & 2, 6, 7, 10--15, 17--23, 25, 29, 30  \\
 & & & 2, 4, 6--8, 10--15, 17--27, 29, 30 \\ \hline
6 & 1 & & 2, 4, 5, 7, 9, 10, 12, 13, 15, 17--23, 25, 26, 28, 29  \\
 & & & 2--5, 7--13, 15--29 \\ \hline
6 & 2 & & 3, 4, 7, 9--11, 14, 15, 17--23, 25, 26, 28--30  \\
 & & & 3, 4, 7, 9--15, 17--23, 25, 26, 28--30 \\ \hline
6 & 3 & & 2, 5, 7, 9--11, 13, 14, 17, 19--23, 25, 26, 28--30  \\
 & & & 2, 4, 5, 7--17, 19--30 \\ \hline
6 & 4 & & 2, 3, 5, 10--15, 18, 19, 21--23, 25, 26, 28--30 \\
 & & & 2, 3, 5, 8, 10--16, 18, 19, 21--30 \\ \hline
6 & 5 & & 3, 4, 7, 9--15, 17, 18, 20, 21, 23, 25, 26, 29, 30  \\
 & & & 2--4, 7--21, 23--30 \\ \hline
6 & 6 & & 2, 4, 5, 7, 9--15, 17--20, 22, 23, 25, 26, 28--30  \\
 & & & 2, 4, 5, 7--23, 25--30 \\ \hline
6 & 7 & & 3, 5, 9--15, 17--19, 21, 22, 25, 28--30  \\
 & & & 2--5, 8--25, 27--30 \\ \hline
6 & 8 & & 3, 4, 7, 10--15, 17--23, 26, 29, 30 \\
 & & & 3, 4, 7, 9--16, 18--27, 29, 30 \\ \hline
6 & 9 & & 2, 4, 5, 7, 10, 11, 13--15, 17--23, 25, 26, 28, 29  \\
 & & & 2--5, 7, 8, 10--29 \\ \hline
7 & 1 & & 3--6, 9, 11--15, 17--23, 25, 26, 28--30  \\
 & & & 2--6, 8, 9, 11--15, 17--22, 24--30 \\ \hline
7 & 2 & & 3, 4, 9--12, 14, 15, 17, 19--23, 25, 26, 28--30  \\
 & & & 3--5, 8--17, 19--30 \\ \hline
7 & 3 & & 2, 4--6, 10--15, 17--19, 21--23, 25, 26, 28--30  \\
 & & & 2, 4, 5, 6, 8, 10--19, 21--30 \\ \hline
7 & 4 & & 2, 3, 6, 9, 10, 12--15, 17--21, 23, 25, 26, 28--30 \\
 & & & 2, 3, 6, 8--10, 12--21, 23--30 \\ \hline
7 & 5 & & 2, 4, 6, 9--15, 17, 18, 20--23, 25, 26, 28--30  \\
 & & & 2--4, 6, 9--23, 25--30 \\ \hline
7 & 6 & & 2, 3, 5, 9--15, 17, 19--23, 25, 28--30  \\
 & & & 3, 5, 8--17, 19--25, 27--30 \\ \hline
7 & 7 & & 2--6, 9--15, 17--23, 25, 26, 29, 30 \\
 & & & 2--6, 8--27, 29, 30 \\ \hline
7 & 8 & & 2, 4--6, 10--15, 17--23, 25, 26, 28, 29  \\
 & & & 2, 4--6, 9--17, 19--29 \\ \hline
7 & 9 & & 2, 3, 5, 10--15, 17, 18, 21--23, 25, 26, 28--30  \\
 & & & 2--5, 8, 10--18, 20-30 \\ \hline
8 & 1 & & 2--7, 9--11, 13, 15, 17, 19, 20, 22, 23, 25, 26, 28--30 \\
 & & & 2--5, 7, 9--17, 19, 20, 22--30 \\ \hline
8 & 2 & & 3--6, 9, 10, 12--15, 17--19, 21--23, 25, 26, 28--30 \\
 & & & 3, 5, 6, 9, 10, 12--19, 21--30 \\ \hline
8 & 3 & & 2, 4--7, 9--15, 17, 20, 21, 23, 25, 26, 28--30  \\
 & & & 2, 4--7, 9--18, 20, 21, 23--30 \\ \hline
8 & 4 & & 5, 9, 11, 12, 14, 15, 17, 19, 21--23, 25, 26, 28--30  \\
 & & & 2, 3, 7, 9, 11, 12, 14--17, 19--23, 25--30 \\ \hline
8 & 5 & & 2--4, 6, 7, 9--15, 18, 20, 21, 23, 25, 28--30  \\
 & & & 2--4, 6, 7, 9--16, 18--25, 27--30 \\ \hline
8 & 6 & & 4, 7, 10--15, 17, 19--23, 25, 26, 29, 30  \\
 & & & 2--4, 7, 10--15, 17--27, 29, 30 \\ \hline
8 & 7 & & 2--6, 9--14, 17--23, 25, 28, 29  \\
 & & & 2--6, 9--14, 16--29 \\ \hline
8 & 8 & & 2, 3, 5--7, 9, 11--13, 15, 17, 18, 20--23, 25, 29, 30 \\
 & & & 2, 3, 5, 6, 9, 11--13, 15--18, 20--25, 27--30 \\ \hline
8 & 9 & & 2--7, 10--12, 14, 15, 17--23, 25, 29 \\
 & & & 2--7, 10--12, 14--30 \\ \hline
9 & 1 & & 2--7, 10--12, 14, 15, 17--19, 21--23, 25, 26, 28--30  \\
 & & & 2--8, 10--12, 14--19, 21--30 \\ \hline
9 & 2 & & 3, 5, 7, 10--15, 17--21, 23, 25, 26, 28--30  \\
 & & & 3--7, 10--21, 23--30 \\ \hline
9 & 3 & & 4--7, 10, 11, 13--15, 17, 19--23, 25, 26, 28--30  \\
 & & & 2, 4--8, 10, 11, 13--23, 25--30 \\ \hline
9 & 4 & & 2, 3, 5, 6, 10--14, 17--23, 25, 28--30  \\
 & & & 2, 3, 5, 6, 8, 10--14, 16--25, 27--30 \\ \hline
9 & 5 & & 2--4, 6, 7, 10, 12--15, 17--23, 25, 26, 29, 30  \\
 & & & 2--4, 6--8, 10, 12--27, 29, 30 \\ \hline
9 & 6 & & 2, 3, 5, 10--15, 17--23, 25, 26, 28, 29  \\
 & & & 2--5, 7, 8, 10--29 \\ \hline
9 & 7 & & 2, 3, 5, 6, 11--15, 17--23, 25, 26, 28--30 \\
 & & & 2--6, 8, 11--30 \\ \hline
9 & 8 & & 2--4, 6, 7, 11, 13--15, 17--23, 25, 26, 28--30  \\
 & & & 2--4, 6, 7, 10, 11, 13--19, 21--30 \\ \hline
9 & 9 & & 2--7, 10--15, 17--23, 25, 26, 28--30 \\
 & & & 2--8, 10--30 \\ \hline
\end{supertabular}
\end{center}

\bigskip
\begin{center}
\tablefirsthead{%
      \multicolumn{4}{c}{TABLE 12}\\
      \multicolumn{4}{c}{Values of $q$-power free constants $2\leq k\leq30$ for which the equation} \\ 
      \multicolumn{4}{c}{$G_n=kx^q$ has no solutions with $A=3$, $B=1$, and $q=3$, $5$} \\ 
      \multicolumn{4}{c}{}\\
      \hline
      $G_0$ & $G_1$ & $q=3$ & \\ 
                    &              & $q=5$ &  \\ \hline\hline}
\tablehead{%
      \hline
\multicolumn{4}{|l|}{Table 12: Impossible values of $k$ in sequences with $A=3$, $B=1$}\\
\multicolumn{4}{|l|}{\small\sl continued from previous page}\\
      \hline
      $G_0$ & $G_1$ & $q=3$ & \\ 
                    &              & $q=5$ &  \\ \hline\hline}
\tabletail{%
      \hline
      \multicolumn{4}{|r|}{\small\sl continued on next page} \\
      \hline}
\tablelasttail{%
      \hline}
\begin{supertabular}{| c | c | c l |}
1 & 1 & & 3, 5, 6, 9--12, 14, 15, 17--22, 25, 26, 28--30  \\
 & & & 3, 5, 6, 8--12, 14--22, 24--30 \\ \hline
1 & 2 & & 3, 5, 6, 9--12, 14, 15, 17--22, 25, 26, 28--30 \\
 & & & 3, 5, 6, 8--12, 14--22, 24--30 \\ \hline
1 & 5 & & 3, 4, 6, 9--15, 18--23, 25, 26, 28--30  \\
 & & & 3, 4, 6--15, 18--30 \\ \hline
1 & 6 & & 2, 4, 5, 7, 9--15, 17, 18, 20--23, 25, 28--30 \\
 & & & 2, 4, 5, 7, 9--18, 20--26, 28---30 \\ \hline
1 & 7 & & 2, 3, 5, 6, 9, 10, 12--15, 17--21, 23, 26, 28--30  \\
 & & & 2, 3, 5, 6, 8--10, 12--21, 23--30 \\ \hline
1 & 8 & & 2, 3, 4, 6, 7, 9--13, 15, 17--23, 26, 28--30  \\
 & & & 2--4, 6, 7, 9--13, 15, 17--24, 26--30 \\ \hline
1 & 9 & & 2--5, 7, 10--15, 18--22, 25, 26, 29, 30 \\
 & & & 2--5, 7, 8, 10--16, 18--27, 29, 30 \\ \hline
2 & 2 & & 3, 5, 6, 7, 9--13, 15, 17, 18, 20--23, 25, 28--30  \\
 & & & 3, 5--7, 9--13, 15--25, 27--30 \\ \hline
2 & 3 & & 4--7, 9, 10, 12--15, 17--23, 25, 26, 28--30  \\
 & & & 4--9, 10, 12--30 \\ \hline
2 & 4 & & 3, 5--7, 9---13, 15, 17, 18, 20--23, 25, 28--30 \\
 & & & 3, 5--7, 9--13, 15--25, 27--30 \\ \hline
2 & 5 & & 3, 4, 6, 9--15, 18--23, 25, 26, 28--30  \\
 & & & 3, 4, 6--15, 18--30 \\ \hline
2 & 9 & & 4--6, 10, 11, 13--15, 17--23, 25, 26, 28, 30  \\
 & & & 3, 4, 6--9, 10--15, 18--30 \\ \hline
3 & 3 & & 2, 4, 5, 7, 9--11, 13--15, 17--20, 22, 23, 25, 26, 28--30  \\
 & & & 2, 4, 5, 7--11, 13--20, 22--30 \\ \hline
3 & 4 & & 2, 7, 9--14, 17, 19--23, 25, 26, 28--30 \\
 & & & 2, 6--14, 16, 17, 19--30 \\ \hline
3 & 5 & & 2, 7, 9--14, 17, 19--23, 25, 26, 28--30  \\
 & & & 2, 6--14, 16, 17, 19--30 \\ \hline
3 & 6 & & 2, 4, 5, 7, 9--11, 13--15, 17--20, 22, 23, 25, 26, 28--30  \\
 & & & 2, 4, 5, 7--11, 13--20, 22--30 \\ \hline
3 & 7 & & 4--6, 10, 11, 13--15, 17--23, 25, 26, 28, 30 \\
 & & & 4--6, 8, 10--23, 25--28, 30 \\ \hline
3 & 8 & & 2, 4, 5, 7, 9--15, 17, 18, 20--23, 25, 28--30  \\
 & & & 2, 4, 5, 7, 9--18, 20--26, 28--30 \\ \hline
4 & 2 & & 3, 5--7, 9, 11--13, 15, 17--23, 25, 26, 28--30 \\
 & & & 3, 5--9, 11--30 \\ \hline
4 & 3 & & 2, 5--7, 10--12, 14, 15, 18--23, 25, 26, 28--30  \\
 & & & 2, 5--8, 10--12, 14--30 \\ \hline
4 & 4 & & 3, 5--7, 9--15, 17--23, 25, 26, 29, 30  \\
 & & & 2, 3, 5--7, 9--15, 17--27, 29, 30 \\ \hline
4 & 5 & & 2, 3, 6, 9--15, 17, 18, 20--23, 26, 28--30 \\
 & & & 2, 3, 6, 8--18, 20--24, 26--30 \\ \hline
4 & 6 & & 2, 3, 5, 7, 10--15, 17--21, 23, 25, 26, 28--30  \\
 & & & 2, 3, 5, 7--21, 23--30 \\ \hline
4 & 7 & & 2, 3, 6, 9--15, 17, 18, 20--23, 26, 28--30 \\
 & & & 2, 3, 6, 8--18, 20--24, 26--30 \\ \hline
4 & 8 & & 3, 5--7, 9--15, 17--23, 25, 26, 29, 30  \\
 & & & 2, 3, 5--7, 9--15, 17--27, 29, 30 \\ \hline
4 & 9 & & 2, 5--7, 10--12, 14, 15, 18--23, 25, 26, 28--30 \\
 & & & 2, 5--8, 10--12, 14--30 \\ \hline
5 & 2 & & 3, 4, 6, 7, 9, 10, 12, 14, 15, 17--23, 25, 26, 28--30  \\
 & & & 3, 4, 6--10, 12, 14--30 \\ \hline
5 & 3 & & 2, 4, 6, 7, 9--11, 15, 17--23, 25, 26, 28--30  \\
 & & & 2, 4, 6--11, 13, 15--30 \\ \hline
5 & 4 & & 2, 3, 6, 7, 9, 10, 12--15, 18--23, 25, 26, 28--30 \\
 & & & 2, 3, 6--10, 12--16, 18--30 \\ \hline
5 & 5 & & 2--4, 6, 7, 9, 11--15, 17--19, 21--23, 25, 26, 28--30 \\
 & & & 2--4, 6--9, 11--19, 21--30 \\ \hline
5 & 6 & & 2, 3, 7, 10--15, 17--22, 25, 26, 28--30 \\
 & & & 2--4, 7, 8, 10--22, 24--30 \\ \hline
5 & 7 & & 2--4, 6, 9--15, 17--23, 25, 28, 30  \\
 & & & 2--4, 6, 9--25, 27, 28, 30 \\ \hline
5 & 8 & & 2--4, 6, 9--15, 17--23, 25, 28, 30  \\
 & & & 2, 3, 4, 6, 9--25, 27, 28, 30 \\ \hline
5 & 9 & & 2, 3, 7, 10--15, 17--22, 25, 26, 28--30 \\
 & & &  2--4, 7, 8, 10--22, 24--30 \\ \hline
6 & 2 & & 3--5, 7, 9--11, 13--15, 17--23, 25, 26, 28--30 \\
 & & & 3--5, 7--11, 14, 15, 17--30 \\ \hline
6 & 3 & & 2, 4, 5, 7, 9--14, 17--20, 22, 23, 25, 26, 28--30 \\
 & & & 2, 4, 5, 7--14, 16--30 \\ \hline
6 & 4 & & 2, 5, 7, 9--13, 15, 17, 19--23, 25, 26, 28--30  \\
 & & & 2, 3, 5, 7--13, 15--17, 19--30 \\ \hline
6 & 5 & & 2--4, 7, 9--12, 14, 15, 17--20, 22, 23, 25, 26, 28--30  \\
 & & & 2--4, 7--12, 14--20, 22--30 \\ \hline
6 & 6 & & 2, 4, 5, 7, 9--11, 13--15, 17--23, 25, 26, 28--30  \\
 & & & 2--5, 7--11, 13--23, 25--30 \\ \hline
6 & 7 & & 3--5, 9, 10, 12--15, 17--23, 25, 26, 28--30  \\
 & & & 2, 3, 5, 8,--10, 12--26, 28--30 \\ \hline
6 & 8 & & 2--5, 7, 9, 11, 13--15, 17--23, 25, 26, 28, 29  \\
 & & & 2--5, 7, 9, 11--29 \\ \hline
6 & 9 & & 2, 3, 5, 7, 10--15, 17--23, 25, 26, 28--30 \\
 & & & 2--5, 7, 8, 10--30 \\ \hline
7 & 1 & & 2--6, 9, 11--15, 17--19, 21--23, 25, 26, 28--30  \\
 & & & 2--6, 8, 9, 11--19, 21--30 \\ \hline
7 & 2 & & 3--6, 9--12, 14, 15, 18, 20--23, 25, 26, 28--30  \\
 & & & 3--6, 8, 9--12, 14--18, 20--30 \\ \hline
7 & 3 & & 4--6, 9--15, 17, 19--23, 25, 26, 28--30  \\
 & & & 2, 4--6, 8--15, 17, 19--30  \\ \hline
7 & 4 & & 2, 3, 5, 6, 9--15, 18, 20, 21, 23, 25, 26, 28--30 \\
 & & & 2, 3, 5, 6, 8--16, 18, 20--30 \\ \hline
7 & 5 & & 3, 4, 6, 9--15, 17--21, 23, 25, 26, 28--30  \\
 & & & 2--4, 6, 8--15, 17--21, 23--30 \\ \hline
7 & 6 & &  2, 4, 5, 9--14, 17--23, 26, 28--30 \\
 & & & 2, 3, 5, 8--14, 16--24, 26--30 \\ \hline
7 & 7 & & 2--6, 9--13, 15, 17--23, 25, 26, 29, 30  \\
 & & & 2--6, 8--13, 15--27, 29, 30 \\ \hline
7 & 8 & & 2--6, 9--12, 14, 15, 17--23, 25, 26, 28--30  \\
 & & & 2--6, 9--12, 14--30 \\ \hline
7 & 9 & & 2--6, 10, 11, 13--15, 17--23, 25, 26, 28--30  \\
 & & & 2--6, 8, 10, 11, 13--30 \\ \hline
8 & 1 & & 2--7, 9, 10, 12--15, 17--22, 25, 26, 28--30 \\
 & & & 2--7, 9, 10, 12--22, 24--30 \\ \hline
8 & 2 & & 3--7, 9--13, 15, 17--21, 23, 25, 26, 28--30  \\
 & & & 3--7, 9--13, 15--21, 23--30 \\ \hline
8 & 3 & & 4--7, 9--15, 18--20, 22, 23, 25, 26, 28--30  \\
 & & & 2, 4--6, 7, 9--16, 18--20, 22--30 \\ \hline
8 & 4 & & 2, 3, 5--7, 9, 10--15, 17--19, 21--23, 25, 26, 29, 30  \\
 & & & 3, 5, 6, 9--19, 21--30 \\ \hline
8 & 5 & & 2--4, 6, 7, 9--13, 15, 17, 18, 20--22, 25, 26, 29, 30  \\
 & & & 2--4, 6, 7, 9--18, 20--22, 24--30 \\ \hline
8 & 6 & & 2--5, 7, 9--15, 17, 19--23, 25, 28--30  \\
 & & & 2--5, 7, 9--17, 19--25, 27--30 \\ \hline
8 & 7 & & 2--6, 9--15, 18--23, 25, 26, 28, 30  \\
 & & & 2--6, 9--16, 18--28, 30 \\ \hline
8 & 8 & & 3, 5, 6, 9--12, 14, 15, 17--22, 25, 26, 28--30  \\
 & & & 2--7, 9--15, 17, 18, 20--30 \\ \hline
8 & 9 & & 2--7, 10--14, 17--23, 25, 26, 28--30 \\
 & & & 2--7, 10--14, 16--30 \\ \hline
9 & 1 & & 2--7, 10, 11, 13--15, 17--23, 25, 28--30  \\
 & & & 2--8, 10, 11, 13--25, 27--30 \\ \hline
9 & 2 & & 3--7, 10--14, 17--23, 26, 28--30  \\
 & & & 3--6, 8, 10--14, 16--24, 26--30 \\ \hline
9 & 3 & & 2, 4--6, 10--15, 17, 19--23, 25, 26, 28--30 \\
 & & & 2, 4--8, 10--17, 19--23, 25--30 \\ \hline
9 & 4 & & 2, 3, 5--7, 10--15, 17--20, 22, 25, 26, 28--30  \\
 & & & 2, 3, 5--8, 10--20, 22, 24--30 \\ \hline
9 & 5 & & 2, 6, 10--15, 17--21, 23, 25, 26, 28--30 \\
 & & & 2--4, 6--8, 10--21, 23, 25--30 \\ \hline
9 & 6 & & 2--5, 7, 10--15, 17--20, 22, 23, 25, 26, 28, 30 \\
 & & & 2--5, 7, 8, 10--20, 22--26, 28--30 \\ \hline
9 & 7 & & 2--5, 10--15, 17--19, 21--23, 25, 26, 28, 29  \\
 & & & 2--6, 8, 10--19, 21--29 \\ \hline
9 & 8 & & 2--7, 10--15, 17, 18, 20--23, 25, 26, 28--30 \\
 & & & 2--7, 10--18, 20--30 \\ \hline
9 & 9 & & 2--7, 10--15, 17, 19--23, 25, 26, 28--30  \\
 & & & 2--8, 10--17, 19--30 \\ \hline
\end{supertabular}
\end{center}


\bigskip

\bigskip

\begin{flushleft}

%
Teresa BOGGIO\\ 
Andrea~MORI\\
Dipartimento di Matematica\\
Universit\`a degli Studi di Torino\\
via Carlo Alberto, 10\\
I-10123 Torino, ITALY\\
e-mail: \texttt{teresa.boggio@gmail.com}\\
e-mail: \texttt{andrea.mori@unito.it}\\[2ex]

\end{flushleft}

\end{document}